\documentclass[12pt]{article}
\usepackage{amssymb,amsmath,amsthm,latexsym}
\usepackage[cp1251]{inputenc}
\usepackage{dsfont}
\usepackage{amssymb}
\usepackage{indentfirst}
\usepackage{graphicx}
\usepackage{caption}
\usepackage{multicol}
\usepackage{multirow}
\usepackage{enumitem}
\usepackage{hyperref}
\hypersetup{citecolor=red,}
\usepackage{amsfonts, array, hhline}
\usepackage{color}
\usepackage{lmodern}
\usepackage[left=1in,right=1in,top=1in,bottom=1.5in]{geometry}

\makeatletter
\theoremstyle{definition}
\newtheorem*{rep@theorem}{\rep@title}
\newcommand{\newreptheorem}[2]{%
\newenvironment{rep#1}[1]{%
 \def\rep@title{#2 \ref{##1}}%
 \begin{rep@theorem}}%
 {\end{rep@theorem}}}
\makeatother

\numberwithin{equation}{section}

\theoremstyle{definition}
\newtheorem{dfn}{Definition}[section]

\newtheorem{thm}[dfn]{Theorem}
\newtheorem{lm}[dfn]{Lemma}

\newtheorem{crl}[dfn]{Corollary}

\newreptheorem{lm}{Lemma}

\theoremstyle{remark}
\newtheorem{rmk}[dfn]{Remark}
\newtheorem{cla}{Claim}[dfn]

\DeclareMathOperator{\arcosh}{arcosh}

\newcommand{\R}{\mathbb{R}}
\renewcommand{\H}{\mathbb{H}}

\newcommand{\E}{\mathbb{E}}
\renewcommand{\S}{\mathbb{S}}

\newcommand{\vol}{{\rm vol}}

\newcommand{\area}{{\rm area}}

\renewcommand{\tilde}{\widetilde}

\title{Prescribed Curvature Problem for Discrete Conformality on Convex Spherical Cone-Metrics}
\author{Ivan Izmestiev,~~~~Roman Prosanov\thanks{This research of R.P. was funded in whole by the Austrian Science Fund (FWF) ESPRIT grant ESP-12-N.},~~~~Tianqi Wu\thanks{The research of T.W. is funded by the U.S. National Science Foundation 1760471.}}
\date{}

\AtEndDocument{\bigskip{\footnotesize
\par
  \textsc{Technische Universit\"at Wien, Wiedner Hauptstra{\ss}e 8-10/104, A-1040 Wien, Austria} \par
  \textit{E-mail}: \texttt{izmestiev@dmg.tuwien.ac.at}
}}

\AtEndDocument{\bigskip{\footnotesize
\par
  \textsc{Universit\"at Wien, Fakult\"at f\"ur Mathematik, Oskar-Morgenstern-Platz 1, A-1090 Wien, Austria} \par
  \textit{E-mail}: \texttt{roman.prosanov@univie.ac.at}
}}

\AtEndDocument{\bigskip{\footnotesize
\par
  \textsc{Clark University, 950 Main St, Worcester, MA 01610, USA} \par
  \textit{E-mail}: \texttt{tianwu@clarku.edu}
}}

\begin{document}
\maketitle

\begin{abstract}
Let $S$ be the 2-sphere and $V \subset S$ be a finite set of at least three points. We show that for each function $\kappa: V \rightarrow (0, 2\pi)$ satisfying elementary necessary conditions, in each discrete conformal class of spherical cone-metrics there exists a unique metric realizing $\kappa$ as its discrete curvature. This can be seen as a discrete version of a result of Luo and Tian.
\end{abstract}

\section{Introduction}

\subsection{Prescribing curvature on surfaces}

A natural question in differential geometry of surfaces is: \emph{which smooth functions on a surface $S$ can appear as the Gaussian curvature of a Riemannian metric $g$ on $S$?} For simplicity we restrict the exposition here to closed surfaces. A basic necessary condition is given by the Gauss--Bonnet theorem, which states that
\begin{equation}
\label{gb}
\int_S \kappa_g d\area_g=2\pi\chi(S).
\end{equation}
Here $\kappa_g$ is the curvature of $g$ and $d\area_g$ is its area form. It implies that a potential function $\kappa$ somewhere must attain a value of the same sign as the Euler characteristic $\chi(S)$.

An early insight was that this problem has too much flexibility, and actually becomes more tractable when we consider not all Riemannian metrics, but restrict ourselves to some \emph{conformal class}.

\begin{dfn}
\label{smoothdef}
Two Riemannian metrics $g$ and $g'$ on $S$ are \emph{simply conformally equivalent} if for a smooth function $u: S \rightarrow \R$ and a diffeomorphism $\phi: S \rightarrow S$ we have $g'=e^{2u}\phi^*(g).$

They are called \emph{conformally equivalent} if $\phi$ is the identity.
\end{dfn}

Note that some authors say conformal equivalence for the former definition, and say pointwise conformal equivalence for the latter.

The prescribed curvature problem in a given conformal class becomes equivalent to a resolution of a non-linear elliptic partial differential equation. When one works in a given simple conformal class, one may also greatly change the potential curvature function acting on it by diffeomorphisms. 

A fundamental special case is given by the celebrated Koebe--Poincar\'e uniformization theorem~\cite{SG}, which states in our language that every Riemannian metric is simply conformally equivalent to a unique up to scaling metric of constant curvature. Due to~(\ref{gb}), the sign of the curvature is the same as the sign of the Euler characteristic. Moreover, the claim can be strengthened to the conformal equivalence except the uniqueness claim in the case of 2-sphere. The 2-sphere is the only surface, for which the conformal group is bigger than the isometry group of a constant curvature metric, so the uniqueness in a conformal class holds up to M\"obius transformation.

The prescribed curvature problem in a given simple conformal class for surfaces was fully solved by Kazdan and Warner~\cite{KW1, KW2}, who showed that the sign condition dictated by (\ref{gb}) is the sufficient condition. They also gave necessary and sufficient conditions for a resolution of this problem in a given conformal class for non-positive Euler characteristic. 
The case of the conformal class for the standard 2-sphere is known as \emph{the Nirenberg problem} and remains to be unresolved in general despite many partial progress, we refer to~\cite{And} as to one of the most recent advancements. 
The uniqueness of a solution in a conformal class is an interesting problem. We mention that in the case of a negative Euler characteristic and a negative curvature function, the solution is unique due to the maximum principle~\cite{KW1}.

A significant amount of work was done to understand the prescribed curvature problem for so-called \emph{cone-metrics} on surfaces. We do not give a general intrinsic definition, but they can be thought as gluings of triangles, which are endowed with a Riemannian metric and have geodesic boundary away from the vertices. The resulting surface may have isolated singularities at the vertices if the total angle turns out to be different from $2\pi$. Such a singular point is called a \emph{cone-point}. The metric is smooth away from cone-points and one can define its (simple) conformal class in the same way as above. Cone-points are atomic points of the curvature measure with the value $2\pi$ minus the total angle. For a cone-metric $d$ with the set of cone-points $V$ we say that its \emph{discrete curvature} $\kappa(d) \in (-\infty, 2\pi)^V$ is the data of the curvatures at the cone-points.

Of the main interest are cone-metrics of constant curvature outside the cone-points, which were actively studied by Alexandrov in order to describe intrinsically the geometry of polyhedral surfaces in 3-dimensional space-forms, see~\cite{Ale}. Since then they became a central object in the field of \emph{discrete differential geometry}. After scaling we may assume that the curvature outside the cone-points is 1, 0 or $-1$. We naturally call such cone-metrics \emph{spherical, Euclidean} and \emph{hyperbolic} respectively.

Most of the study of the prescribed curvature problem for cone-metrics deals with prescribing the discrete curvature and looking for constant curvature outside the singularities. If we fix a conformal structure outside of singularities (but extending over the singularities), then this can be seen as the uniformization of punctured surfaces with prescribed behaviour at the punctures. The interest to this setting goes back to Picard (see, e.g.,~\cite{Pic}), who investigated some cases of the prescribed discrete curvature problem in a conformal class of hyperbolic cone-metrics. A complete result (see~\cite{Hei, McO}) states that if $V \subset S$ is a finite set and $\kappa: V \rightarrow (-\infty, 2\pi)$ is a function such that
$$\sum_{v \in V}\kappa_v>2\pi\chi(S),$$
then in each conformal class on $S \backslash V$, extending to a conformal structure over $S$, there exists a unique hyperbolic cone-metric realizing $\kappa$ as its discrete curvature. A similar result for Euclidean cone-metrics was established by Troyanov~\cite{Tro1}. In~\cite{Tro2} Troyanov also studied some cases when the curvature varies away from cone-points.

The curvature prescription problem for spherical cone-metrics exhibits curious difficulties, especially when $S$ is the 2-sphere. In the latter case it turns out that the Gauss--Bonnet condition is not the only necessary condition. For instance, in the case of \emph{convex} spherical cone-metrics (i.e., when the discrete curvature $\kappa \in (0, 2\pi)^V$) 
Luo and Tian~\cite{LT} proved the following result:

\begin{thm}
\label{LT}
Let $S$ be the 2-sphere, $V \subset S$ be a finite set of at least three points, $S$ be endowed with a conformal structure and $\kappa$ be in $(0,2\pi)^V$. Then there exists a spherical cone-metric $d$ on $(S,V)$ such that

(a) the conformal structure induced by $d$ on $S\backslash V$ coincides with the given one, and

(b) $d$ has the discrete curvature $\kappa$,

if and only if
\begin{equation}
\label{e1}
\sum_{v\in V}\kappa_v<4\pi
\end{equation}
and
\begin{equation}
\label{e2}
\kappa_v<\sum_{w\neq v}\kappa_w\text{, for any $v\in V$}.
\end{equation}
Further, such $d$ is unique if it exists.
\end{thm}

The existence part of Theorem~\ref{LT} was also established in~\cite{Tro2}. The current paper is devoted to a discrete version of Theorem~\ref{LT} formulated in the next subsection. 

The necessity of condition~(\ref{e1}) is given by the Gauss--Bonnet theorem. We can, however, see that it is not the only condition already on the example of a spherical (geodesic) triangle: together with the Gauss-Bonnet condition, its exterior angles also obey the triangle inequality, as they are the side-lengths of the polar triangle. (It is interesting to observe how this situation contrasts the Euclidean and the hyperbolic settings.) The condition~(\ref{e2}) is a manifestation of this effect, and can be easily proven by induction on the number of cone-points. (The induction step can be done by cutting the metric along a geodesic segment connecting two cone-points, and gluing there a bigon with one cone-point in its interior.)

In the non-convex case it was not until recently when Mondello and Panov~\cite{MP1} obtained nearly complete necessary and sufficient conditions on the curvatures of spherical cone-metrics on the 2-sphere (the study of last remaining cases was completed by Eremenko~\cite{Ere1}). They also treated the case of surfaces of higher genus in~\cite{MP2}, where they interestingly showed that the Gauss--Bonnet condition is actually sufficient.

It is worth to mention that a solution of the prescribed discrete curvature problem in a simply conformal class of spherical cone-metrics on the 2-sphere is non-unique in general. It is conjectured, however, that the number of solutions is finite. We refer to the recent survey of Eremenko~\cite{Ere}. A well-studied case is when all $\kappa_v(d)$ are multiples of $2\pi$, so the metric is a branched covering of the standard 2-sphere. The answer on the number of solutions for a generic position of points was given by Scherbak~\cite{Sch}. 

\subsection{Discrete conformality of cone-metrics}

The purpose of discrete differential geometry is to discretize the notions from smooth differential geometry in order to obtain analogues of classical results. Ideally, the existence results in the discrete setting should imply the existence results in the smooth setting by means of some kind of approximation. It frequently happens that the discretization of various equivalent notions gives rise to non-equivalent discrete objects.

Discretization of conformal geometry has a long history mainly focused on conformal maps between domains in the complex plane, especially on establishing a discrete version of the Riemann mapping theorem. The most well-known approach is via \emph{circle packings} as proposed by Thurston. However, 
circle packings do not work that well for discretizing the notion of conformal equivalence between Riemannian metrics on surfaces, since it was often unclear how to construct a circle packing on a general Riemannian surface.

Let $S$ be a closed surface, $d$ be a Euclidean cone-metric on $S$, $V$ be a finite point set containing all cone-points of $d$ and $\mathcal T$ be a geodesic triangulation of $d$ with vertices at $V$. (In what follows we will say that $d$ is a Euclidean cone-metric on $(S, V)$ realizing $\mathcal T$.) Then $d$ is uniquely determined by the lengths of the edges of $\mathcal T$. One can attempt to discretize the notion of pointwise conformal equivalence by considering metrics $d'$ defined by the same data $(S, V, \mathcal T)$ and satisfying
\begin{equation}
\label{conftemp}
l_e(d')=\exp\left(\frac{u_v+u_w}{2}\right)l_e(d),
\end{equation}
for each edge $e$ of $\mathcal T$ with endpoints $v$ and $w$, where $l_e(d)$, $l_e(d')$ are the lengths of $e$ in $d$, $d'$ and $u: V \rightarrow \R$ is some function. One can say that such $d'$ is \emph{discretely conformally equivalent} to $d$ \emph{with respect to $\mathcal T$}. This is motivated by the fact that if a Riemannian metric tensor changes by $e^{2u}$, then the distances change due to~(\ref{conftemp}) up to the third order.

This definition of discrete conformality first occurred in~\cite{RW} in an attempt to discretize some aspects of the theory of relativity. It was rediscovered by Luo~\cite{Luo} in his development of a combinatorial Yamabe flow. In~\cite{BPS} Bobenko, Pinkall and Springborn showed a surprising connection of this discrete conformal change with the change of decorations on some associated hyperbolic cusp-metric. Finally, in~\cite{GLSW} Gu, Luo, Sun and Wu gave a novel definition of discrete conformality, which extends (\ref{conftemp}), but does not employ an arbitrary triangulation. 

Each Euclidean cone-metric has a canonical \emph{Delaunay decomposition} into convex polygons, which is generically a triangulation. The \emph{discrete conformal class} in the sense of~\cite{GLSW} is generated by conformal changes of the type (\ref{conftemp}) performed on Delaunay triangulations, and by diagonal flips among Delaunay triangulations for metrics, which admit several of them. (An exact definition is given in Definition~\ref{confdefe}.) An evidence of the validity of this definition is given by the fact that it is coherent with the connection developed in~\cite{BPS}: all metrics obtained this way have the same associated hyperbolic cusp-metric. In the present paper we work in the setting of discrete conformality proposed in~\cite{GLSW}.

In~\cite{GLSW} the authors resolved completely the discrete curvature prescription problem in the new setting, demonstrating the same behaviour as in the smooth setting. They showed

\begin{thm}
\label{pce}
Let $S$ be a closed surface, $V \subset S$ be a finite set of points, $d$ be a Euclidean cone-metric on $(S, V)$ and $\kappa \in (-\infty, 2\pi)^V$ be a function satisfying 
$$\sum_{v \in V}\kappa_v=2\pi\chi(S).$$
Then there exists a unique up to scaling Euclidean cone-metric $d'$ on $(S, V)$ that is discretely conformal to $d$ and has the discrete curvature $\kappa$.
\end{thm}

In a sequel paper~\cite{GGLSW} a similar result for hyperbolic cone-metrics was also established:

\begin{thm}
\label{pch}
Let $S$ be a closed surface, $V \subset S$ be a finite set of points, $d$ be a hyperbolic cone-metric on $(S, V)$ and $\kappa \in (-\infty, 2\pi)^V$ be a function satisfying 
$$\sum_{v \in V}\kappa_v>2\pi\chi(S).$$
Then there exists a unique hyperbolic cone-metric $d'$ on $(S, V)$ that is discretely conformal to $d$ and has the discrete curvature $\kappa$.
\end{thm}

In particular, the discrete uniformization theorems follow for surfaces of non-positive Euler characteristic: in each discrete conformal class of hyperbolic (resp. Euclidean) cone-metrics on a surface of genus at least two (resp. genus one) there exists a unique metric (resp. unique up to scaling) without cone-singularities.

In subsequent papers~\cite{GLW, LSW, LWZ, WZ} various types of convergence were examined. It was shown that this version of discrete uniformization allows to compute conformal maps in the Riemannian mapping theorem and the uniformization factors for Riemannian surfaces provided that the approximating triangulations are chosen carefully enough.

As outlined in~\cite{BPS}, the discrete uniformization in this setting is connected with the geometry of convex ideal polyhedral surfaces in the hyperbolic 3-space. More generally, it turns out that this notion of discrete conformality has a nice geometric interpretation in terms of hyperbolic cone-polyhedra. This viewpoint was explored in~\cite{Pro1}, where an alternative variational proof of Theorem~\ref{pch} was given. It was based on some previous papers~\cite{BI, Izm2, FI1, FI2}, which used cone-polyhedra to obtain results on convex realizations of cone-metrics and their rigidity. One can see a deep connection between these topics, which is further developed in the present paper. Another paper working out this connection is the recent paper~\cite{Spr} of Springborn. Its results imply, in particular, the discrete uniformization theorem for spherical cone-metrics on the 2-sphere, although it is not explicitly stated in the text. In~\cite{Nie2} Nie gives another 3-dimensional interpretation of discrete conformality by generalizing and exploring the classical Epstein--Penner convex hull construction.

Except the discrete uniformization nothing is currently known about discrete conformality of spherical cone-metrics. By analogy with the smooth setting, one should naturally expect a more complicated behavior from this case. In this paper we make a step further and establish a discrete analogue of Theorem~\ref{LT}.

\begin{thm}
\label{main}
Let $S$ be the 2-sphere, $V \subset S$ be a finite set of at least three points, $ d$ be a spherical cone-metric on $(S,V)$ that admits a triangulation into convex spherical triangles, and
$\kappa\in(0,2\pi)^V$.
Then there exists a spherical cone-metric $d'$ on $(S,V)$ that is discretely conformal to $d$ and has the discrete curvature $\kappa$, if and only if (\ref{e1}) and (\ref{e2}) hold.
Further, such $d'$ is unique if it exists.
\end{thm}

By a convex spherical triangle we mean isometric to a geodesic triangle in the standard sphere with all angles in $(0, \pi)$. The condition on admitting a triangulation into convex spherical triangles is required to well-define a discrete conformal class. See the next section for more discussion on geodesic triangulations of spherical cone-metrics. We also note that due to the Gauss--Bonnet theorem, the 2-sphere is the only orientable surface admitting convex spherical cone-metrics.

Theorem~\ref{main} is proven by a continuity method popularized by Alexandrov, see~\cite{Ale}. We parametrize the discrete conformal class of $d$ by some domain $U(d) \subset \R^V$ and consider the curvature as a $C^1$-map $\kappa: U(d) \rightarrow (-\infty, 2\pi)^V$. Let $U_C(d) \subset U(d)$ be the subset of convex cone-metrics. We show that, when restricted to $U_C(d)$, the differential of $\kappa$ is non-degenerate, the map is proper and the domain is connected. Since the target is clearly simply connected, we obtain that $\kappa|_{U_C(d)}$ is a homeomorphism onto its image, which is the set of $\kappa\in(0,2\pi)^V$ satisfying (\ref{e1}) and (\ref{e2}).

This method was also used in the proofs of Theorems~\ref{pce} and~\ref{pch} in~\cite{GLSW, GGLSW}. The non-degeneracy of the differential is frequently based on the fact that it is equal to the Hessian of an auxiliary functional $H$. The main difference of our case from the previously resolved ones lies in the fact that in the hyperbolic and Euclidean settings the respective functional is strictly concave, which immediately implies the non-degeneracy. This behavior of the functional no longer holds in the spherical case, which brings many complications and is responsible for some flexibility phenomena specific to this case. 

In order to obtain the non-degeneracy in our situation, we give $H$ a geometric interpretation in terms of ideal cone-polyhedra, similarly as it was done in~\cite{Pro1}. This interpretation allows us to use dimensional-reduction arguments inspired by~\cite{BI, Izm2}.

A very interesting question is what is going on without the convexity assumption. First of all, if all $\kappa_v(u)$ are zero, that is, we have just the standard sphere with marked points, then $d\kappa$ has a 3-dimensional kernel. It comes from the possibility to move the interior point inside our cone-polyhedron. The kernel consists of the linear functions on the associated fan. This actually corresponds to a local non-rigidity of such ``metrics'' caused by M\"obius transformations that are not isometries. The same happens with branched coverings of the standard sphere, i.e., the case when all $\kappa_v$ are integer multiples of $2\pi$. 

We were unable to construct more non-trivial examples of global or local non-rigidity, although we conjecture that they exist, similarly to the smooth setting. We, however, detected an infinitesimal non-rigidity, which is not caused by the effects above.

\begin{table}
\centering
\scalebox{0.8}{
\begin{tabular}{|c||c|c|c|c|c|c|c|}
\hline
Inner angle & $2\pi/5$ & $2.4\pi/5$ & $2.8\pi/5$ & $3.2\pi/5$ & $3.6\pi/5$ & $4\pi/5$ & $4.4\pi/5$ \\
\hline
\hline
\multirow{12}{*}{Eigenvalues} & $-2.7751$ & $-8.2113$ & $-16.6180$ & $-29.1759$ & $-48.3430$ & $-80.5748$ & $-149.9030$ \\ 
\cline{2-8}
& $0.0000$ & $-3.3463$ & $-8.6839$ & $-16.7566$ & $-29.1077$ & $-49.7980$ & $-93.9818$ \\ 
\cline{2-8}
& $0.0000$ & $-3.3463$ & $-8.6839$ & $-16.7566$ & $-29.1077$ & $-49.7980$ & $-93.9818$ \\ 
\cline{2-8}
& $0.0000$ & $-3.3463$ & $-8.6839$ & $-16.7566$ & $-29.1077$ & $-49.7980$ & $-93.9818$ \\ 
\cline{2-8}
& $3.2492$ & $2.3498$ & $0.6055$ & $-2.2159$ & $-6.5865$ & $-13.7638$ & $-28.5081$ \\ 
\cline{2-8}
& $3.2492$ & $2.3498$ & $0.6055$ & $-2.2159$ & $-6.5865$ & $-13.7638$ & $-28.5081$ \\ 
\cline{2-8}
& $3.2492$ & $2.3498$ & $0.6055$ & $-2.2159$ & $-6.5865$ & $-13.7638$ & $-28.5081$ \\ 
\cline{2-8}
& $3.2492$ & $2.3498$ & $0.6055$ & $-2.2159$ & $-6.5865$ & $-13.7638$ & $-28.5081$ \\ 
\cline{2-8}
& $3.2492$ & $2.3498$ & $0.6055$ & $-2.2159$ & $-6.5865$ & $-13.7638$ & $-28.5081$ \\ 
\cline{2-8}
& $4.4903$ & $4.5256$ & $4.1538$ & $3.3382$ & $2.0158$ & $0.0000$ & $-3.4994$ \\
\cline{2-8}
& $4.4903$ & $4.5256$ & $4.1538$ & $3.3382$ & $2.0158$ & $0.0000$ & $-3.4994$ \\
\cline{2-8}
& $4.4903$ & $4.5256$ & $4.1538$ & $3.3382$ & $2.0158$ & $0.0000$ & $-3.4994$ \\
\hline
\end{tabular}
}
\caption{Eigenvalues of the Jacobian of $d\kappa$ for regular icosahedron-type metrics.}
\label{icosah}
\end{table}

Let $(S, V, \mathcal T)$ be the combinatorial structure of an icosahedron. Consider a family of spherical cone-metrics $d_t$ assigning equal lengths to all edges. We parametrize them by the value of an inner angle in a triangle of $\mathcal T$. We computed the eigenvalues of the Jacobian matrix of $d\kappa$ for some $d_t$, see Table~\ref{icosah}. We see the expected 3-dimensional kernel for the angles $2\pi/5$ and $4\pi/5$. Except this, a 5-dimensional kernel appears for some angle between $2.8\pi/5$ and $3.2\pi/5$. It is curious if there is a geometric explanation to this phenomenon.

If there are actual examples of non-uniqueness not coming from M\"obius transformations, we conjecture, similarly to the smooth setting, that the number of such solutions is finite. Finally, we emphasize that nothing is currently known about discrete conformality of spherical cone-metrics on closed surfaces other than the 2-sphere or the projective plane (which then necessarily contain a cone-point of negative discrete curvature). This is an interesting direction of future research.

We also would like to mention the related recent works~\cite{Nie, GHZ} considering the prescription problem of the so-called \emph{geodesic curvature} for circle patterns in spherical cone-metrics. It is interesting to note that the functional appearing in that problem is strictly convex, similarly to the Euclidean and hyperbolic cases. 

Another current direction of research in discrete conformality is developing the theory of \emph{decorated discrete conformal equivalence}, which incorporates together the vertex-scaling approach with the inversive distance circle packings. The work~\cite{BL} particularly establishes this theory in the context of spherical geometry.
\vskip+0.2cm

\textbf{Acknowledgments.} We thank the anonymous referees for useful comments.

\section{Preliminaries}

\subsection{Cone-metrics and Delaunay triangulations}

Suppose that $S$ is a closed orientable surface and $V \subset S$ is a finite point-set.

\begin{dfn}
\label{conedfn}
A \emph{spherical (resp. Euclidean) cone-metric} $d$ on $(S, V)$ is a metric on $S$ locally isometric to the standard sphere (resp. the Euclidean plane) except points of $V$. At $v \in V$ the metric $d$ is locally isometric to the metric of a spherical (resp. Euclidean) cone with total angle $\omega_v(d)$, which might be distinct from $2\pi$. A cone-metric $d$ is called \emph{convex} if for every $v \in V$ we have $\omega_v(d) \leq 2\pi$. 
Two cone-metrics on $(S, V)$ are considered equivalent if there is an isometry fixing $V$ and isotopic to the identity with respect to $V$.
\end{dfn}

The number $2\pi-\omega_v(d)$ is called the \emph{curvature} of $v$ in $d$. We will denote it by $\kappa_v(d)$ in the case of spherical cone-metrics and by $\delta_v(d)$ in the case of Euclidean cone-metrics. In this paper we are primarily interested in spherical cone-metrics, but at some steps we need also to use their connection with Euclidean cone-metrics. 

By $\mathfrak M_+(S,V)$ and $\mathfrak M_0(S, V)$ we denote the sets of equivalent classes of spherical and Euclidean cone-metrics respectively. In what follows by a cone-metric we sometimes mean an equivalence class of cone-metrics. 

We say that a spherical triangle is convex if any inner angle is in $(0,\pi)$ or, equivalently, the triangle is contained in an open hemisphere. Under this assumption, a spherical triangle is determined up to isometry by its edge lengths, and the space of possible edge lengths is
$$
\{(l_1,l_2,l_3)\in (0, \pi)^3 :l_1+l_2+l_3<2\pi, \text{$(l_1,l_2,l_3)$ satisfies the triangle inequality}\}.
$$
In particular, the length of each edge is less than $\pi$.

\begin{dfn}
\label{celldef}
A \emph{polygonal decomposition} of $(S, V)$ is a collection of simple disjoint arcs (called  \emph{edges}) with endpoints in $V$ that cut $S$ into topological polygons (called \emph{faces}) without points of $V$ in the interior. We require that each polygon has at least three vertices after being cut (although some of them may coincide as actual points of $V$). Two cell decompositions on $(S, V)$ are considered equivalent if they are isotopic with respect to~$V$.

A \emph{triangulation} of $(S, V)$ is a polygonal decomposition such that all faces are triangles.
\end{dfn}

Note that the definition above allows loops and multiple edges. We denote the set of edges of a triangulation $\mathcal T$ by $E(\mathcal T)$. In what follows by a polygonal decomposition or by a triangulation we sometimes mean an equivalence class of them. 

Given a cone-metric $d$ on $(S,V)$, we say that a triangulation $\mathcal T$ of $(S,V)$ is \emph{realized} by $d$ if there exists a geodesic triangulation of $(S, V)$ equivalent to $\mathcal T$ (which is then clearly unique). In the case of spherical cone-metrics we additionally assume that the geodesic triangulation consists only of convex triangles. 

For a triangulation $\mathcal T$ of $(S, V)$ by $\mathfrak M_+(\mathcal T) \subset \mathfrak M_+(S,V)$ and $\mathfrak M_0(\mathcal T) \subset \mathfrak M_0(S,V)$ we denote the subsets of cone-metrics realizing $\mathcal T$. Any cone-metric in $\mathfrak M_+(\mathcal T)$, $\mathfrak M_0(\mathcal T)$ is uniquely determined by its type (spherical or Euclidean), $\mathcal T$ and the edge lengths. Thus, $\mathfrak M_+(\mathcal T)$, $\mathfrak M_0(\mathcal T)$ can be parametrized as open convex polyhedra in $\R^{E(\mathcal T)}$:
\begin{multline*}
\mathfrak M_+(\mathcal T)=\{
l\in \mathbb (0,\pi)^{E(\mathcal T)}:\text{$l$ satisfies the triangle inequalities},\\ \text{the perimeter of each triangle is $<2\pi$}
\},
\end{multline*}
\begin{equation*}
\mathfrak M_0(\mathcal T)=\{
l\in \mathbb (0, \infty)^{E(\mathcal T)}:\text{$l$ satisfies the triangle inequalities}
\}.
\end{equation*}

\begin{rmk}
Not any spherical cone-metric $d$ on $(S,V)$ has a geodesic triangulation with vertex set $V$. For example, if $(S,d)$ is the unit sphere and $V$ is contained in an open hemisphere, then $(S,V,d)$ does not admit a geodesic triangulation. If $V$ belongs to a great circle, then $(S, V, d)$ admits a triangulation, but not into convex triangles. On the other hand, it is not hard to show that if for every point of $S$ there exists a point of $V$ at distance less than $\pi/2$, then $d$ admits a triangulation into convex triangles. For this one can adapt arguments from \cite[Proposition 3.1]{Thu}, \cite[Section 3]{Riv}.
\end{rmk}

By $\mathfrak M_{+, T}(S, V) \subset \mathfrak M_+(S, V)$ we denote the subset of spherical cone-metrics on $(S, V)$ admitting a geodesic triangulation into convex triangles. Due to the remark above,  $$\mathfrak M_{+, T}(S, V) \subsetneq \mathfrak M_+(S, V).$$ On the other hand, any Euclidean triangle is considered to be convex, and it is well-known that any Euclidean cone-metric on $(S, V)$ admits a geodesic triangulation. 

The families $\{\mathfrak M_+(\mathcal T)\}_{\mathcal T}$ and $\{\mathfrak M_0(\mathcal T)\}_{\mathcal T}$ provide open coverings of the spaces $\mathfrak M_{+, T}(S, V)$ and $\mathfrak M_0(S, V)$ respectively, and endow them with the structure of smooth manifolds. The discrete curvatures $\kappa$ and $\delta$ can be considered as smooth $\mathbb R^V$-valued functions on $\mathfrak M_{+, T}(S, V)$ and $\mathfrak M_0(S, V)$ respectively. 

We note that there are infinitely many isotopy classes of triangulations of $(S, V)$. Moreover, a Euclidean cone-metric may realize infinitely many triangulations: consider, e.g., the surface of the Euclidean unit cube. However, for spherical metrics the convexity condition does not allow this.

\begin{lm}
\label{finite}
Each metric $d \in \mathfrak M_{+, T}(S, V)$ realizes finitely many triangulations. 
\end{lm}

\begin{proof}
It is enough to show that between any two points $v, w \in V$ there are finitely many simple geodesic arcs of length $< \pi$. To this purpose we adapt the argument from~\cite[Proposition 1]{ILTC}. 

Note that any two different simple geodesic arcs of length $<\pi$ are non-isotopic on $S \backslash V$, so they must bound together a region with at least one point of $V$. Similarly, none of them is isotopic to an arc of length $\pi$. The family of all simple geodesic arcs from $v$ to $w$ of length at most $\pi$ is uniformly Lipschitz. Thus, by the Arzel\`a--Ascoli theorem it is compact. Hence, if there are infinitely many arcs of length $<2\pi$, we can extract a convergent subsequence. But we saw that no arc can appear in the limit, so we get a contradiction. 
\end{proof}

Of special importance among triangulations are the well-known \emph{Delaunay triangulations}.

\begin{dfn}
We say that a triangulation $\mathcal T$ realized by a spherical (resp. Euclidean) cone-metric $d$ is \emph{Delaunay} in $d$ if when we develop each pair of the adjacent triangles to the standard sphere (resp. the Euclidean plane), then the circumscribed disk of each triangle does not contain the opposite vertex of the other triangle in the interior.
\end{dfn}

A Delaunay triangulation is a refinement of a more general \emph{Delaunay decomposition}. 

\begin{dfn}
A \emph{Delaunay decomposition} of a spherical (resp. Euclidean) cone-metric on $(S, V)$ is a polygonal decomposition such that each cell is isometric to a circumscribed convex polygon in the standard sphere $\S^2$ (resp. in the Euclidean plane $\E^2$) and for each pair of adjacent cells after we develop them in $\S^2$ (resp. $\E^2$) the closed circumscribed disk of each polygon does not contain the vertices of the other one except the vertices of the adjacent edge. In the spherical case by a \emph{convex polygon} we mean that all angles are at most $\pi$ and it is contained in an open hemisphere.
\end{dfn}

So in particular, if we take a Delaunay triangulation and erase all the edges such that adjacent triangles form a cyclic quadrilateral, then we obtain a Delaunay decomposition. Conversely, any triangulation refining a Delaunay decomposition is Delaunay. It is well-known that a Delaunay decomposition exists for every Euclidean cone-metric and is actually unique, see~\cite{BS, ILTC}. We need to check that this is also true for spherical cone-metrics.

\begin{lm}
\label{delaun}
Let $d \in \mathfrak M_{+, T}(S, V)$. A Delaunay decomposition of $d$ exists and is unique.
\end{lm}

Instead of adapting the arguments from~\cite{BS}, we will prove Lemma~\ref{delaun} in Section~\ref{conpol} with the help of our main tool: \emph{ideal cone-polyhedra}.

Note that although a Euclidean cone-metric may realize infinitely many triangulations, it is clear from the uniqueness of the Delaunay decomposition that it has only finitely many Delaunay triangulations. Moreover, both in Euclidean and in spherical cases any two Delaunay triangulations can be connected by a sequence of flips via Delaunay triangulations.

%

By $\mathfrak D_+(\mathcal T)\subset \mathfrak M_+(\mathcal T)$, $\mathfrak D_0(\mathcal T)\subset \mathfrak M_0(\mathcal T)$ we denote the spaces of cone-metrics realizing $\mathcal T$ as a Delaunay triangulation. We have the decompositions

$$
\mathfrak M_{+, T}(S,V)=\bigcup_{\mathcal T}\mathfrak D_+( \mathcal T),~~~~~\mathfrak M_0(S,V)=\bigcup_{\mathcal T}\mathfrak D_0( \mathcal T).
$$
One can see that these decompositions are locally finite, each cell $\mathfrak D_+( \mathcal T)$ or $\mathfrak D_0( \mathcal T)$ is closed, has non-empty interior and piecewise-smooth boundary, and any two cells have disjoint interiors. 

%
%
%
%

\subsection{Discrete conformal equivalence}

\begin{dfn}
\label{confdefs}
We say that $d$ and $d'$ in $\mathfrak M_{+, T}(S, V)$ are \emph{discretely conformally equivalent} if there are a sequence of metrics $d=d_1,d_2,\ldots,d_m=d'$ and a sequence of triangulations $ \mathcal T_1,\ldots, \mathcal T_m$ of $(S,V)$ such that

(a) $d_i\in\mathfrak D_+(\mathcal T_i)$, and

(b) if $\mathcal T_i= \mathcal T_{i+1}$, then there is a function $u: V \rightarrow \R$ such that for each edge $e$ of $\mathcal T$ with endpoints $v$ and $w$ we have
\begin{equation}
\label{conformals}
\sin\left(\frac{l_e(d_{i+1})}{2}\right)=\exp\left(\frac{u_v+u_w}{2}\right)\sin\left(\frac{l_e(d_i)}{2}\right),
\end{equation}
where $l_e(d)$ is the length of $e$ in $d$; and

(c) if $\mathcal T_i\neq \mathcal T_{i+1}$, then $d_i=d_{i+1}$.
\end{dfn}

We denote the discrete conformal class of $d \in \mathfrak M_{+, T}(S, V)$ by $\mathfrak C_+(d)$.

\begin{dfn}
\label{confdefe}
We say that $d$ and $d'$ in $\mathfrak M_0(S, V)$ are \emph{discretely conformally equivalent} if there are a sequence of metrics $d=d_1,d_2,\ldots,d_m=d'$ and a sequence of triangulations $ \mathcal T_1,\ldots, \mathcal T_m$ of $(S,V)$ such that

(a) $d_i\in\mathfrak D_0(\mathcal T_i)$, and

(b) if $\mathcal T_i= \mathcal T_{i+1}$, then there is a function $u: V \rightarrow \R$ such that for each edge $e$ of $\mathcal T$ with endpoints $v$ and $w$ we have
\begin{equation}
\label{conformale}
l_e(d_{i+1})=\exp\left(\frac{u_v+u_w}{2}\right)l_e(d_i),
\end{equation}
where $l_e(d)$ is the length of $e$ in $d$; and

(c) if $\mathcal T_i\neq \mathcal T_{i+1}$, then $d_i=d_{i+1}$.
\end{dfn}

We denote the discrete conformal class of $d \in \mathfrak M_0(S, V)$ by $\mathfrak C_0(d)$.

Let $d \in \mathfrak D_0(\mathcal T)$. Then clearly for any real $\lambda > 0$ the metric $d_{\lambda}$ obtained from $d$ by multiplying all the edge lengths of $\mathcal T$ by $\lambda$ belongs to $\mathfrak D_0(\mathcal T) \cap \mathfrak C_0(d)$. This defines an equivalence relation on $\mathfrak C_0(d)$. We denote the set of equivalence classes by $\tilde{\mathfrak C}_0(d)$. It is clear that for all $\lambda > 0$ we have $\delta(d_{\lambda})=\delta(d)$. Thus, we may consider $\delta$ as a function over $\tilde{\mathfrak C}_0(d)$.

Note that due to the Gauss--Bonnet theorem, the image of $\delta$ lies in the set
$$\Delta=\{\delta \in (-\infty, 2\pi)^V: \sum_{v \in V}\delta_v = 4\pi \} \subset \R^V.$$
The main result of~\cite{GLSW} is the following.
\begin{thm}
\label{glsw}
$\mathfrak C_0(d)$ is $C^1$-diffeomorphic to $\mathbb R^V$, and
the map
$$\delta: \tilde{\mathfrak C_0}(d) \rightarrow \Delta$$
is a $C^1$-diffeomorphism. 
\end{thm}

Let $\mathcal T$ be a triangulation of $(S, V)$. Given $l\in\mathbb R^{E(\mathcal T)}$, we define $\sin(l/2)\in\mathbb R^{E(\mathcal T)}$ as
$\sin(l/2)_e=\sin (l_e/2)$ for each $e \in E(\mathcal T)$. Let $T$ be a convex spherical triangle and $l \in (0, \pi)^{E(T)}$ be the 3-tuple of its side lengths. Consider it on the unit 2-sphere $\S^2$ embedded in the Euclidean 3-space. Then the Euclidean triangle subtending $T$ on $\S^2$ has lengths $2\sin(l/2)$. This means that if $d \in \mathfrak D_+(\mathcal T)$ and $l \in (0, \pi)^{E(\mathcal T)}$ is the vector of its edge lengths, then $2\sin(l/2)$ determines a Euclidean cone-metric, which we denote by $\Phi_{\mathcal T}(d)$. This defines a map
$$\Phi_{\mathcal T}: \mathfrak D_+(\mathcal T) \rightarrow \mathfrak M_0(\mathcal T).$$
The image of $\Phi_{\mathcal T}$ consists exactly of those cone-metrics in $\mathfrak M_0(\mathcal T)$ having the circumradius of each triangle of $\mathcal T$ less than 1.

By this construction, it is easy to see that if $d\in\mathfrak M_{+,T}(S,V)$ realizes more than one Delaunay triangulations, the image of $d$ 
does not depend on the choice of the map $\Phi_{\mathcal T}$. So the maps $\Phi_{\mathcal T}$ give an injective map
$$\Phi: \mathfrak M_{+,T}(S, V) \rightarrow \mathfrak M_0(S, V).$$

Consider two adjacent triangles in the unit sphere. They satisfy the Delaunay condition if and only if the dihedral angle between the subtending Euclidean triangles is at most $\pi$ when measured inwards the sphere. By unbending them we see that this corresponds to the Delaunay condition between the Euclidean triangles. This shows

\begin{lm}
\label{phi1}
The metric $d \in \mathfrak D_+(\mathcal T)$ if and only if $\Phi(d) \in \mathfrak D_0(\mathcal T)$.
\end{lm}

From this and the definitions we get

\begin{lm}
\label{phi2}
Two metrics $d$ and $d'$ are discretely conformally equivalent if and only if $\Phi(d)$ is discretely conformally equivalent to $\Phi(d')$.
\end{lm}

Now we show
\begin{lm}
\label{connects}
The space $\mathfrak C_+(d)$ is a connected manifold of dimension $|V|$ without boundary.
\end{lm}

\begin{proof}
Since $\mathfrak C_0(\Phi(d))$ is homeomorphic to $\mathbb R^V$ by Theorem \ref{glsw} and $\Phi$ is continuous and injective, 
it suffices to show that $\Phi(\mathfrak C_+(d))$ is open and connected in $\mathfrak C_0(\Phi(d))$. $\Phi(\mathfrak C_+(d))$ contains all the metric $d'\in \mathfrak C_0(\Phi(d))$ where
the radius of any circumcircle of the Delaunay decomposition of $d'$ is smaller than $1$. So the openness is obvious.
Let $d', d'' \in \mathfrak C_+(d)$, we can connect $\Phi(d')$ and $\Phi(d'')$ by a path $d_t \subset \mathfrak C_0(\Phi(d))$. 
As the decomposition of $\mathfrak C_0$ induced by cells  $\mathfrak D_0(\mathcal T)$ is piecewise-smooth and locally finite, we can make $d_t$ intersecting the boundaries of this decompostion only finitely many times.
If we multiply a Euclidean cone-metric by $\lambda>0$, we stay in the same discrete conformal class and in the same cells. Thus, we can multiply all $d_t$ by sufficiently small $\lambda$ so that all the circumscribed circles of the Delaunay decompositions of $\lambda\cdot d_t$ become smaller than 1. By linearly connecting the corresponding end points of $d_t$ and $\lambda\cdot d_t$, we
obtain a path in $\Phi(\mathfrak C_+(d))$ connecting $\Phi(d')$ and $\Phi(d'')$.

\end{proof}

\subsection{Cusp-metrics}
In the study of the discrete conformality, it was necessary to introduce hyperbolic cusp-metrics. See also \cite{BPS, GLSW, GGLSW, Spr} for employing this connection to study the discrete conformality of Euclidean and hyperbolic cone-metrics.

\begin{dfn}
A \emph{hyperbolic cusp-metric} $g$ on $(S, V)$ is a complete hyperbolic metric of finite area on $S \backslash V$. 
\end{dfn}

Points of $V$ are called \emph{cusps} in a hyperbolic cusp-metric.

\begin{dfn}
A \emph{decoration} $h$ of $(S, V)$ with a hyperbolic cusp-metric $g$ is a choice of a horocycle at each cusp.
\end{dfn}

After we chose a decoration, we can speak about the distance between two cusps: we consider the signed distance between the corresponding horocycles taken with the minus sign if they intersect.
Although $g$ is actually a metric on the surface with punctures $S \backslash V$, we can still speak about realizations of triangulations with vertices at $V$. The edges are infinite geodesics approaching the cusps like ideal points of $\overline \H^2$. It is well-known that

\begin{lm}[\cite{Mar}, Proposition 7.4.6]
For each hyperbolic cusp-metric $g$ on $(S, V)$ each triangulation $\mathcal T$ of $(S, V)$ is realized by $g$.
\end{lm}

The reader can find more details about cusp-metrics in~\cite{Pen}. The basic aspects are also treated well in~\cite[Chapter 7.4]{Mar}.

In the work~\cite{GLSW} a discrete conformal class in the Euclidean case $\mathfrak C_0(d)$ was associated with a hyperbolic cusp-metric $g$. It was shown that a Delaunay triangulation $\mathcal T$ of a metric $d' \in \mathfrak C_0(d)$ is a Delaunay triangulation of $g$ for some choice of decoration $h$, also called an \emph{Epstein--Penner triangulation}. Here we do not give the definition of a Delaunay triangulation of a decorated hyperbolic cusp-metric, but refer to~\cite{Spr, EP}. The work of Akiyoshi~\cite{Aki} implies then

\begin{lm}
Let $d \in \mathfrak M_0(S, V)$. There are finitely many triangulations $\mathcal T$ of $(S, V)$ such that $\mathfrak D_0(\mathcal T)\cap \mathfrak C_0(d)$ is non-empty.
\end{lm}

Together with Lemmas~\ref{phi1} and~\ref{phi2} this implies

\begin{crl}
\label{fin}
Let $d \in \mathfrak M_{+, T}(S, V)$. There are finitely many triangulations $\mathcal T$ of $(S, V)$ such that $\mathfrak D_+(\mathcal T)\cap \mathfrak C_+(d)$ is non-empty.
\end{crl}

\subsection{Semi-ideal triangles and tetrahedra}
\label{laws}

We will use some basic trigonometry of hyperbolic triangles having simultaneously ideal and non-ideal vertices. We consider ideal points in $\H^2$ or $\H^3$ always equipped with horocycles or horospheres respectively, called \emph{canonical}. In other words, we abuse the notation and by an ideal point we mean a pair (ideal point, horocycle/horosphere centered at this point). By a distance between two ideal points we mean the signed distance between their canonical horocycles/horospheres. Similarly, the distance from an ideal point to a non-ideal one means the signed distance from the canonical horocycle/horosphere to the non-ideal point.

\begin{dfn}
A triangle $BA_1A_2$ in $\H^2$ is called \emph{semi-ideal} if the vertices $A_1$ and $A_2$ are ideal and the vertex $B$ is not.
\end{dfn}

We denote the distances $BA_1$, $BA_2$, $A_1A_2$ by $u_1$, $u_2$, $a_{12}$ respectively. By $l_{12}$ we denote the angle at $B$. By $\rho_{12}$ and $\rho_{21}$ we denote the lengths of the parts of the canonical horocycles at $A_1$ and $A_2$ respectively cut out by the angles of the triangle. 

\begin{lm}[Cosine laws]
\label{coslaw}
We have
$$\rho^2_{12}=\exp(u_2-u_1-a_{12})-\exp(-2u_1),~~~~~\rho^2_{21}=\exp(u_1-u_2-a_{12})-\exp(-2u_2),$$
$$\cos(l_{12})=\frac{\exp(u_1+u_2)-2\exp(a_{12})}{\exp(u_1+u_2)}.$$
\end{lm}

The first two formulas are shown in~\cite[Lemma 2.3]{Pro1}. The last formula is obtained from the cosine law for hyperbolic triangles with all non-ideal vertices by passing to the limit. We also have an analogue of the sine law:

\begin{lm}[Sine law]
\label{sinlaw}
We have
$$\frac{1}{2}\sin(l_{12})\exp(-a_{12})=\rho_{12}\exp(-u_2)=\rho_{21}\exp(-u_1).$$
\end{lm}

Indeed, it is not hard to see from the cosine laws that the square of each expression is equal to $$\exp(-u_1-u_2-a_{12})-\exp(-2u_1-2u_2).$$

Let $BA_{12}$ be the perpendicular from $B$ to $A_1A_2$. By $b_{12}$ and $b_{21}$ we denote the distances $A_{12}A_1$ and $A_{12}A_2$ respectively, by $u_{12}$ denote the distance from $B$ to $A_1A_2$. (So $a_{12}=b_{12}+b_{21}$.)

\begin{lm}
\label{perpend}
We have
$$\exp(u_1)=\exp(b_{12})\cosh(u_{12}),~~~~~\exp(u_2)=\exp(b_{21})\cosh(u_{12}),$$
$$\rho^2_{12}=\exp(-2b_{12})-\exp(-2u_1),~~~~~\rho^2_{21}=\exp(-2b_{21})-\exp(-2u_2).$$
\end{lm}

Here the first two formulas are just the limits of the Pythagoras theorem for hyperbolic triangles. The last two formulas are obtained from substituting them in the cosine laws.

\begin{dfn}
A tetrahedron $BA_1A_2A_3$ in $\H^3$ is called \emph{semi-ideal} if the vertices $A_1$, $A_2$ and $A_3$ are ideal and the vertex $B$ is not.
\end{dfn}

The notation for the parameters of a semi-ideal tetrahedron $BA_1A_2A_3$ is inherited from semi-ideal triangles.

\subsection{Ideal cone-polyhedra}
\label{conpol}

Let $d \in \mathfrak M_{+}(\mathcal T)$ for some triangulation $\mathcal T$ of $(S, V)$. For each triangle $T$ of $\mathcal T$ consider a semi-ideal tetrahedron $\bar T$ in $\H^3$ that has $T$ as its spherical link at the non-ideal vertex $B$. Here by the spherical link we mean the part of the unit sphere in $T_B\H^3$ that is cut off by our tetrahedron. Note that $\bar T$ exists and is determined up to isometry by $T$. We decorate $\bar T$ by considering at each ideal vertex the horosphere that passes through the non-ideal vertex. Glue all such tetrahedra together with respect to $\mathcal T$ so that at the ideal boundary the canonical horospheres 
match together. We obtain a complex $P=P(d, \mathcal T)$, which is a complete cone-manifold with ideal boundary, called \emph{ideal cone-polyhedron}. We consider them up to marked isometry: isometry such that the induced map on $(S, V)$ fixes $V$ and is isotopic to identity with respect to $V$. We say that the point of $P$ obtained from gluing all the non-ideal vertices of the tetrahedra is \emph{the marked point} of $P$.

We note that the total dihedral angle at the interior edge corresponding to $v \in V$ is $\omega_v(d)$. We can say that $\kappa_v(d)$ is the curvature of this interior edge. For $e \in E(\mathcal T)$ we denote the dihedral angle of an edge $e \in E(\mathcal T)$ at the boundary of $P$ by $\alpha_e$. We say that $P$ is \emph{convex} if all $\alpha_e \leq \pi$. With the help of the central projection from the marked point we identify the boundary of $P$ with $(S, V)$, equipping the latter with a decorated cusp-metric.

We can construct ideal cone-polyhedra in another way. Let $(g, h)$ be a decorated hyperbolic cusp-metric on $(S, V)$, $\mathcal T$ be a triangulation of $(S, V)$ and $u \in \R^V$. Assume that for every triangle of $\mathcal T$ there exists a (non-degenerate) semi-ideal tetrahedron with the ideal face being a decorated ideal triangle coming from the realization of $\mathcal T$ in $(g, h)$ and edge lengths to the non-ideal point being determined by $u$. Then we can glue these tetrahedra together so that the decorations given by $h$ match together. We obtain an ideal cone-polyhedron, which in this case we denote by $P=\bar P(g, h, \mathcal T, u)$.

For an ideal cone-polyhedron $P=P(d,\mathcal T)$ or $P=\bar P(g, h, \mathcal T, u)$ we say that an edge $e$ of $\mathcal T$ is an edge of $P$ if $\alpha_e \neq \pi$. A face of $P$ is a connected component of the boundary of $P$ minus all edges. This determines \emph{the face decomposition of $P$}. A priori there could be non-simply connected faces in the face decomposition, then it would not be a polygonal decomposition in the sense of Definition~\ref{celldef}. We will soon show that this is not the case.

Let $o \in \H^3$, $\Pi \subset \H^3$ be a geodesic plane not containing $o$ and $q$ be the closest point from $\Pi$ to $o$. By $\mu: \Pi \rightarrow \R$ we denote the distance function from $\Pi$ to $o$. Then for each $p \in \Pi$ we have
\begin{equation}
\label{d}
\cosh\mu(p)=\cosh \mu(q)\cosh d_{\H^3}(p,q).
\end{equation}

Now consider an ideal cone-polyhedron $P$. Let $\mu: S\backslash V \rightarrow \R$ be the \emph{distance function} from the boundary of $P$, identified with $S\backslash V$, to the marked point of $P$. Let $\psi: [0, \tau] \rightarrow S$ be a unit-speed geodesic in the cusp-metric of the boundary. Then it may have kink points coming from the edges of $P$. Denote their coordinates by $t_1, \ldots, t_k \in [0, \tau]$, also set $t_0:=0$, $t_{k+1}:=\tau$. Due to~(\ref{d}), the restriction of $\mu\circ \psi$ to $(t_i, t_{i+1})$ has the form 
\begin{equation}
\label{d1}
\mu\circ \psi(t)=\arcosh(b_i\cosh(t-a_i))
\end{equation}
for some real numbers $a_i$, $b_i$ where $i=0, \ldots, k$. If $P$ is convex, then at the kink points $t_i$ the left derivative of $\mu \circ \psi$ is greater than the right derivative. 

\begin{lm}
\label{face}
All faces of an ideal cone-polyhedron $P$ are simply connected.
\end{lm}

\begin{proof}
In the same way as in~\cite[Lemma 4.12]{Pro1} one can show that if we have a non-simply connected face, then it contains a closed geodesic $\psi$ (in the extrinsic sense, so intersecting no edges of $P$). Then the restriction of $\mu$ to $\psi$ must be periodic. On the other hand, it must have the form (\ref{d1}), which is not periodic, hence the contradiction.
\end{proof}

It follows that the face decomposition of $P$ is a polygonal decomposition.

\begin{lm}
\label{delaunay}
The ideal cone-polyhedron $P(d, \mathcal T)$ is convex if and only if $\mathcal T$ is Delaunay for $d$. Moreover, if $e \in E(\mathcal T)$ is an interior edge of a face in the Delaunay decomposition containing $\mathcal T$, then $\alpha_e=\pi$. 
\end{lm}

\begin{proof}
Indeed, the dihedral angle of the intersection of two planes in $\H^3$ is equal to the intersection angle of the two circles in $\partial_{\infty} \H^3$ that are their boundaries at infinity. Clearly, the conformal structure at $\partial_{\infty} \H^3$ coincides with the conformal structure of the unit sphere. This implies the desired statement.
\end{proof}

This shows that the face decomposition of $P(d, \mathcal T)$ is exactly the Delaunay decomposition containing $\mathcal T$.
Now we are ready to show that it exists and is uniquely determined by $d$. 

\begin{proof}[Proof of Lemma~\ref{delaun}.]
Let $\mathcal T$ and $\mathcal T'$ be two different Delaunay triangulations of $d$, $e \in E(\mathcal T)$ and $p \in S$ be its intersection point with an edge $e' \in E(\mathcal T')$ in their realizations in $d$. Due to Lemma~\ref{delaunay}, both ideal cone-polyhedra $P=P(d, \mathcal T)$ and $P'=P(d, \mathcal T')$ are convex.

By $\mu$, $\mu'$ we denote the distance functions from the boundaries of $P$, $P'$ respectively to the marked point where the boundaries are identified with $S \backslash V$. From the convexity of $P'$, by considering a semi-ideal triangle subtended by $e$ one sees that $\mu'(p) \geq \mu(p)$. Similarly, by considering the semi-ideal triangle subtended by $e'$ we see that $\mu(p) \geq \mu(p')$. Thus, $\mu(p)=\mu'(p)$.

The union of $\mathcal T$ and $\mathcal T'$ cuts $(S, d)$ into convex polygons. We can cut it further without adding new vertices to obtain a triangulation. This decomposes $P$, $P'$ into tetrahedra with some ideal vertices and some non-ideal ones, all with a marked non-ideal point. It is easy to see that if two hyperbolic tetrahedra, both with a marked non-ideal point, have isometric spherical links at the marked points and equal distances from these points to the respective non-ideal points, then the tetrahedra are isometric. This allows us to construct a marked isometry between $P$ and $P'$. Thus, both $\mathcal T$ and $\mathcal T'$ refine the face decomposition of $P=P'$. Thus, the Delaunay decomposition of $d$ is unique and is the face decomposition of $P$.

Now we show that the Delaunay decomposition exists. Take any triangulation $\mathcal T$ realized by $d$. If it is not Delaunay, then there exists an edge $e \in E(\mathcal T)$ such that its dihedral angle in $P(d, \mathcal T)$ is greater than $\pi$. It is easy to see that $e$ is adjacent to two distinct triangles and that the union of their realizations in $d$ is a convex quadrilateral. Thus, we can flip $e$. The volume of $P(d, \mathcal T)$ increases under this operation. Due to Lemma~\ref{finite}, there are finitely many triangulations realized by $d$, thus, the algorithm finishes in finitely many steps.
\end{proof}

Let $\mathcal T$ be a Delaunay triangulation for $d \in \mathfrak M_{+}(\mathcal T)$. Due to Lemma~\ref{delaunay}, the ideal cone-polyhedron $P(d, \mathcal T)$ is independent of a particular choice of $\mathcal T$ in case if there are several Delaunay triangulations. We denote it just by $P(d)$ and denote its boundary cusp-metric by $g_d$. It also comes with a canonical decoration $h_d$ of $g_d$.

Now we need to return to the second construction of ideal cone-polyhedra. So take a decorated hyperbolic cusp-metric $(g, h)$ on $(S, V)$. We can show

\begin{lm}
\label{triang}
Let $P=\bar P(g, h, \mathcal T, u)$ and $P'=\bar P(g, h, \mathcal T', u)$ be two convex ideal cone-polyhedra. Then they are marked isometric.
\end{lm}

\begin{proof}
The proof is similar to the proof of Lemma~\ref{delaun} just above, but is more involved as we need to use that the distance functions of $P$ and $P'$ coincide asymptotically at the cusps. This is almost identical to the proof of a similar fact in \cite[Section 4.1]{Pro1}, but we sketch it here for self-completeness. The proof was inspired by \cite[Section 2.3]{BI} or \cite[Section 3.1]{FI1}, but the situation there was dealing with compact boundaries, so without the need to investigate the asymptotic behaviour.

Consider an edge $e$ of $\mathcal T$. We parametrize it by the signed distance along $e$ to the horocycle of $h$ at one of its endpoints $v$. Let $\mu$, $\mu'$ be the distance functions of $P$ and $P'$ respectively and $f$, $F$ be the restrictions of $\cosh \mu$, $\cosh \mu'$ to $e$. 

The function $f$ has the form 
$$f(t)=b\cosh(t-a).$$ 
It is important to note that two functions of this form either coincide or have at most one point in common.
There exist real numbers $t_1< \ldots< t_k \in \R$ such that at $(t_i, t_{i+1})$ the function $F(t)$ has the form
$$F(t)=b_i\cosh(t-a_i).$$
Here we also assume $t_0:=-\infty$, $t_{k+1}:=+\infty$. It might happen that $k=0$.
By $F_i(t)$ we denote the function $b_i \cosh(t-a_i)$ extended to all $\R$. It is clear that for all $i=0, \ldots, k-1$ we have $F_i(t) > F_{i+1}(t)$ for $t> t_{i+1}$ and $F_i(t) < F_{i+1}(t)$ for $t<t_{i+1}$. By induction we get that $F_i(t) > F(t)$ for all $t \notin (t_i, t_{i+1})$.

We prove that $F(t) \geq f(t)$ for all $t \in \R$. We need two simple statements.

\begin{cla}
\label{claim2} 
Let $$f_1(x)=b_1\cosh(x-a_1),$$ $$f_2(x)=b_2\cosh(x-a_2)$$ and for $t \in \R$ we have $f_1(t)=f_2(t)$ and $\dot f_1(t)> \dot f_2(t)$. Then $a_2 > a_1$.
\end{cla}

The proof is an easy computation, see~\cite[Section 4.1]{Pro1}. It then follows by induction that $a_0 < \ldots < a_k$.

\begin{cla}
\label{claim3}
Let $\psi_1$ and $\psi_2$ be two distinct geodesic lines in $\H^2$ meeting at a point $A \in \partial_{\infty} \H^2$ and $o \in \H^2$ be outside of the angle formed by $\psi_1$, $\psi_2$. Let $A$ be decorated by an horocycle and $\psi_1$, $\psi_2$ be parametrized by the signed distance to this horocycle. Denote the hyperbolic cosines of the distance functions from $\psi_1$ and $\psi_2$ to $o$ by $$f_1(t)=b_1\cosh(t-a_1),$$  $$f_2(t)=b_2\cosh(t-a_2)$$ respectively. Then ${f_1(t)-f_2(t)}$ has a constant nonzero sign. Besides, if $f_1(t)>f_2(t)$, then $a_2>a_1$.
\end{cla}

The proof is straightforward. Now consider $F_0(t)$ and $f(t)$. Their difference has a constant sign. Suppose that $F_0(t) < f(t)$. Then Claim~\ref{claim3} implies that $a<a_0$. Thus, we have $a< a_{k}$. Also we see that $f(t)>F(t)$ for all $t \in \R$. Consider now the other parametrization of $e$ by the signed distance along $e$ to the horocycle at its second endpoint. We get the reparametrized functions $f(a_e-t)$, $F(a_e-t)$, where $a_e$ is the length of $e$ in the metric $g$ decorated by $h$. We similarly obtain $a_e -a <a_e - a_{k+1}$, or $a_{k+1}<a$. This is a contradiction, so $F_0(t)\geq f(t)$ for all $t \in \R$. Similarly, $f(t) \leq F_{k+1}(t)$ for all $t \in \R$. 

Now assume that for some $t' \in \R$ we have $f(t') > F(t')$. Then either for each $t > t'$ or for each $t<t'$ we have $f(t) > F(t)$. But this contradicts with $F_0(t) \geq f(t)$, $F_{k+1}(t) \geq f(t)$ for all $t \in \R$.

Thus, for each point $p \in S$ that is an intersection point of two distinct edges $e \in E(\mathcal T)$ and $e' \in E(\mathcal T)$ realized in $(S, g)$ we have $\mu(p) \leq \mu'(p)$. Similarly we can obtain $\mu(p) \geq \mu'(p)$. Hence, $\mu(p)=\mu'(p)$.

The union of $\mathcal T$ and $\mathcal T'$ cuts $(S, g)$ into convex hyperbolic polygons with some vertices possibly at cusps. We can cut it further to a triangulation without adding new vertices. At each vertex, which is not a cusp, the distances in $P$, $P'$ to the marked point coincide. At each cusp the signed distances from the horospheres determined by $h$ to the marked point coincide. One can see that each tetrahedron is determined up to isometry by this data. It follows that $P$ is marked isometric to $P'$.
\end{proof}

Lemma~\ref{triang} implies that if an ideal cone-polyhedron $\bar P(g, h, \mathcal T, u)$ is convex, we can denote it just by $\bar P(g, h, u)$. By $U'(g, h)$ we denote the set of those $u \in \R^V$ that there exists a convex ideal cone-polyhedron $\bar P(g, h, u)$. In the case $g=g_d$, $h=h_d$ for a spherical cone-metric $d$ we also denote $U'(g_d, h_d)$ by $U'(d)$, and denote $\bar P(g_d, h_d, u)$ by $P(d, u)$. 

\begin{lm}
The set $U'(g, h)$ is open.
\end{lm}

\begin{proof}
Let $u \in U'(g, h)$ and $\mathcal T$ be a face triangulation of $\bar P(g,h,u)$. As all semi-ideal tetrahedra determined by $g$, $h$, $\mathcal T$ and $u$ are non-degenerate, they exist also for all $u'$ sufficiently close to $u$. Then the ideal cone-polyhedron $\bar P(g, h, \mathcal T, u')$ is well-defined. However, it can be non-convex. 

Due to Lemma~\ref{face}, there are finitely many face triangulations of $\bar P(g,h,u)$. Choose a neighbourhood $N \ni u$ so that all edges of $\bar P(g,h,u)$ remain to have $\alpha_e<\pi$ also in $\bar P(g, h, \mathcal T', u')$ for all $u' \in N$ and all face triangulations $\mathcal T'$ of $\bar P(g,h,u)$. Then we perform the flip-algorithm for $\bar P(g, h, \mathcal T, u')$ from the proof of Lemma~\ref{delaun}: at each step take a non-convex edge and flip it. The flip is always possible and the volume increases under a flip. Due to our choice of $N$, an edge of $\bar P(g,h,u)$ can never be flipped. Thus, all appearing triangulations are face triangulations of $\bar P(g,h,u)$. Since there are finitely many of them, the algorithm finishes in finitely many steps, providing a triangulation $\mathcal T'$ such that $\bar P(g, h, \mathcal T', u')$ is convex. Hence, $N \subset U'(g,h)$.
\end{proof}

By $U(d)$ we denote the connected component of $U'(d)$ containing $o$. We actually think that $U(g, h)$ is always connected, so $U(d)=U'(d)$. However, we do not need this for the needs of the present paper, so we do not dwell upon this. 

For a triangulation $\mathcal T$ of $(S, V)$  let $U'(g, h, \mathcal T)$ be the set of $u \in U'(g, h)$ such that $\mathcal T$ is isotopic with respect to $V$ to a face triangulation of $\bar P(g,h,u)$. By $U(d,h,\mathcal T)$ we denote the intersection $U'(g,h,\mathcal T)\cap U(g,h)$. So $U(g, h, \mathcal T)$ provide a locally finite decomposition of $U(g, h)$ into cells with piecewise-smooth boundaries. In case of $g=g_d$, $h=h_d$ we denote $U(g_d, h_d, \mathcal T)$ just by $U(d, \mathcal T)$. 

Using geometric arguments similar to the proof of Lemma 4.14 in~\cite{Pro1} one can show that if $u \in U(g, h, \mathcal T)$, then $\mathcal T$ is an Epstein--Penner triangulation for $g$ and a decoration $h'$ having the horocycle at $v$ at the signed distance $u_v$ from the respective horocycle of $h$. Combined with the Akiyoshi result, this would give another proof of Corollary~\ref{fin}. We will not rely on this, so we do not clarify this more.

Consider two metrics $d , d' \in \mathfrak M_{+,T}(S, V)$. Assume that the respective cusp-metrics $g_{d }, g_{d'}$ are isometric by an isometry isotopic to identity with respect to $V$. Let $h_{d }$ be the decoration of $g_{d }$ coming from the ideal cone-polyhedron $P(d )$. Identify $g_{d }$ and $g_{d'}$ with the help of the isometry we have. However, the decoration $h_{d'}$ coming from the ideal cone-polyhedron $P(d')$ is different from $h_{d }$. By $u_{d }(d') \in \R^V$  we denote the signed distances of the horocycles of $h_{d'}$ from the horocycles of $h_{d }$. Clearly, the ideal cone-polyhedron $P(d')$ coincides with the ideal cone-polyhedron $P(d, u_{d }(d')).$ We can show

\begin{lm}
\label{connection}
Two metrics $d , d' \in \mathfrak M_T(S, V)$ are discretely conformally equivalent if and only if the respective cusp-metrics $g_{d }, g_{d'}$ are isometric by an isometry isotopic to the identity with respect to $V$ and $u_{d }(d') \in U(d )$.
\end{lm}

\begin{proof}
Let $u'$ be in $U(d )$ and $d'$ be the spherical cone-metric induced on the spherical link of the marked point in $P(d, u')$. We show that $d'$ is discretely conformally equivalent to $d$. Consider a path $u_t$ connecting $u'$ with the origin $o$. As the decomposition of $U(d)$ into cells $U(d, \mathcal T)$ is locally finite and piecewise-smooth, we can choose $u_t$ so that it intersects the boundaries finitely many times. Let $u_1, \ldots, u_k$ be the intermediate points of the path at the boundaries of $U(d, \mathcal T)$ and $d_1, \ldots, d_k$ be the induced spherical cone metrics. We also set $u_0:=o$, $u_{k+1}:=u'$, $d_0:=d$, $d_{k+1}:=d'$. By construction, each $u_i$ and $u_{i+1}$ belongs to the same $U(d, \mathcal T)$, we denote this $\mathcal T$ by $\mathcal T_i$. Due to Lemma~\ref{delaunay}, $d_i \in \mathfrak D_+(\mathcal T_i)$. So we need to check that $d_i$ and $d_{i+1}$ are discretely conformally equivalent in the sense of (b) from Definition~\ref{confdefs}.

Let $e$ be an edge of $\mathcal T_i$ with endpoints $v$ and $w$. Then, due to Lemma~\ref{coslaw}, we have
$$\cos (l_e(d_i))=1-2\exp(a_e-u_{i,v}-u_{i,w}),$$
$$\cos (l_e(d_{i+1}))=1-2\exp(a_e-u_{i+1,v}-u_{i+1,w}).$$
Here $a_e$ is the length of the edge $e$ in the metric $g_d$ with the decoration $h_d$. We can rewrite it as
$$\sin\left(\frac{l_e(d_i)}{2}\right)=\sqrt{\frac{1-\cos l_e(d_i)}{2}}=\exp\left(\frac{a_e-u_{i,v}-u_{i,w}}{2}\right),$$
$$\sin\left(\frac{l_e(d_{i+1})}{2}\right)=\sqrt{\frac{1-\cos l_e(d_{i+1})}{2}}=\exp\left(\frac{a_e-u_{i+1,v}-u_{i+1,w}}{2}\right),$$
$$\sin\left(\frac{l_e(d_{i+1})}{2}\right)=\exp\left(\frac{u_{i,v}-u_{i+1,v}}{2}+\frac{u_{i,w}-u_{i+1,w}}{2}\right)\sin\left(\frac{l_e(d_{i})}{2}\right).$$
This shows that each $d_i$ and $d_{i+1}$ are discretely conformally equivalent and so are $d$ and $d'$.

Now suppose that we have two discretely conformal metrics $d$ and $d'$. We want to show that $g_d$ is marked isometric to $g_{d'}$. Due to Definition~\ref{confdefs}, the uniqueness of Delaunay decomposition and Lemma~\ref{delaunay}, it is enough to show this under the assumption that both $d, d' \in \mathfrak D_+(\mathcal T)$ for some triangulation $\mathcal T$. Then there exists $u \in \R^V$ such that for each edge $e$ of $\mathcal T$ with endpoints $v$ and $w$ we have
\begin{equation}
\label{t1}
\sin\left(\frac{l_e(d')}{2}\right)=\exp\left(\frac{u_v+u_w}{2}\right)\sin\left(\frac{l_e(d)}{2}\right).
\end{equation}

Note that the lengths $a_e$, $a'_e$ of the edge $e$ in the cusp-metric $g_d$ with the decoration $h_d$ and in $g_{d'}$ with $h_{d'}$ respectively satisfy
\begin{equation}
\label{t2}
\sin\left(\frac{l_e(d)}{2}\right)=\exp\left(\frac{a_e}{2}\right),
\end{equation}
\begin{equation}
\label{t3}
\sin\left(\frac{l_e(d')}{2}\right)=\exp\left(\frac{a'_e}{2}\right).
\end{equation}
For each $v \in V$ choose at the respective cusp of $P(d')$ a horoball so that the marked point is at the signed distance $-u_v$ from it. This determines a decoration $h'$ of $g_{d'}$. It follows from (\ref{t1}), (\ref{t2}) and (\ref{t3}) that the length of $e$ in $g_{d'}$ with the decoration $h'$ is $a_e$. Thus, the pairs $(g_d, h_d)$ and $(g_{d'}, h')$ have the same Penner coordinates. It follows that $g_{d'}$ is marked isometric to $g_d$.

The operation above allows us to define a map from $\mathfrak C_+(d)$ to $U'(d)$, which is continuous. By Corollary~\ref{connects}, $\mathfrak C_+(d)$ is connected. It follows that for every discretely conformal metrics $d$ and $d'$ we have $u_d(d') \in U(d)$.
\end{proof}

%
%
%
%


The construction above produces a continuous map $\Psi_d: U(d) \rightarrow \mathfrak C_+(d)$, which is injective due to Lemmas~\ref{delaunay} and~\ref{delaun}. It was also shown that it has a continuous injective inverse. Hence, $\Psi_d$ is a homeomorphism. One can show that it is $C^1$, but we will only use that the curvature function $\kappa \circ \Psi_d$ is $C^1$. We will obtain this in next subsection. 

\begin{rmk}
\label{inverse}
It follows from the considerations above that if $u, u' \in U(g,h,\mathcal T)$ for some triangulation $\mathcal T$, then the respective spherical cone-metrics $d$, $d'$ satisfy $d, d' \in \mathfrak D_+(\mathcal T)$ and for each $e \in E(\mathcal T)$ with endpoints $v$, $w$ we have
$$\sin\left(\frac{l_e(d')}{2}\right)=\exp\left(\frac{u_v+u_w-u'_v-u'_w}{2}\right)\sin\left(\frac{l_e(d)}{2}\right).$$
\end{rmk}

Also Lemma~\ref{connection} and Corollary~\ref{fin} in turn imply

\begin{crl}
\label{finc}
Let $d \in \mathfrak M_{+, T}(S, V)$. There are finitely many triangulations $\mathcal T$ of $(S, V)$ such that $U(d, \mathcal T)$ is non-empty.
\end{crl}

\subsection{Discrete curvature}

Let $(g, h)$ be a decorated cusp metric. For a convex ideal cone polyhedron $P=\bar P(g, h, u)$ define its \emph{total discrete curvature} $H(P) \in \R$ as
$$H(P)=-2\vol (P)+\sum_{v \in V} u_v \kappa_v + \sum_{e \in E(\mathcal T)} a_e(\pi-\alpha_e).$$
Here $\kappa_v$ are the curvatures of the interior edges of $P$, $\mathcal T$ is a face triangulation of $P$, $\alpha_e$ are the dihedral angles of edges of $P$ and $a_e$ are their lengths in $g$ with $h$.

This defines the \emph{discrete curvature functional} $H$ over $U'(g,h)$. It is also sometimes called \emph{the discrete Hilbert--Einstein functional}. Here the sum $-2\vol (P)+\sum_{v \in V} u_v \kappa_v$ can be regarded as the discrete version of the integral of the scalar curvature, and the sum  $\sum_{e \in E(\mathcal T)} a_e(\pi-\alpha_e)$ is the discrete version of the integral of the mean curvature of the boundary. This functional was first introduced by Volkov in his PhD thesis from 1955 to give a variational proof of the Alexandrov realization theorem, see~\cite{Vol1, Vol2}. It was notably used in various similar problems since then, see~\cite{Izm} for a survey and~\cite{BI, FI1, FI2, Pro1, Pro2} for examples of its use.

\begin{lm}
\label{firstder}
The functional $H$ is $C^1$ and we have
$$\frac{\partial H}{\partial u_v}=\kappa_v.$$
\end{lm}

\begin{proof}
The proof basically is an application of the celebrated Schl\"affli formula, which was adapted to partially ideal hyperbolic polyhedra in~\cite[Theorem 14.5]{Riv2}.

Consider a decorated semi-ideal tetrahedron $\bar T$ with the edge lengths of the ideal face $a_1$, $a_2$, $a_3$, their dihedral angles $\alpha_1$, $\alpha_2$, $\alpha_3$ respectively, the lengths of semi-ideal edges $u_1$, $u_2$, $u_3$ and their dihedral angles $\omega_1$, $\omega_2$, $\omega_3$ respectively. Then the Schl\"affli formula says that
$$-2 d\vol(\bar T)=\sum u_i d\omega_i+\sum a_i d\alpha_i.$$

Suppose that $u$ is in the interior of a cell $U(g, h, \mathcal T)$. Then the combinatorics does not change locally around $P$ and, as the dihedral angles are smooth functions of the edge lengths of semi-ideal polyhedra, we see that $H$ is smooth. Summing all the equalities for particular tetrahedra we get
$$-2d\vol(P)=-\sum_{v\in V} u_v d\kappa_v+\sum_{e\in \mathcal T} a_e d\alpha_e.$$
This implies 
$$dH(P)=\sum_{v \in V}\kappa_v d u_v,$$
which is the desired statement.

When $u$ is not in the interior of $U(g,h,\mathcal T)$, we recall that the decomposition of $U(g,h)$ into cells $U(g,h,\mathcal T)$ is finite and the cells have piecewise-smooth boundaries. Thus, one can compute directional derivatives at $u$ with the help of the method above. One sees then that all the partial derivatives
$$\frac{\partial H}{\partial u_v}=\kappa_v
$$
exist and are continuous around $u$. This implies that $H$ is $C^1$. 
\end{proof}

Now we are going to show that $H$ is actually $C^2$ and to compute its second derivatives. By $\vec E_v(\mathcal T)$ we denote the set of oriented edges of $\mathcal T$ emanating from $v \in V$. By $\vec E_{vw}(\mathcal T)$ we denote the set of oriented edges emanating from $v \in V$ and ending at $w \in V$. For a convex ideal cone-polyhedron $P=\bar P(g, h, u)$ and an oriented edge $\vec e$ of its face triangulation we denote by $\alpha^+_{\vec e}$ and by $\alpha^-_{\vec e}$ the dihedral angles of the two tetrahedra adjacent to $\vec e$. By $\rho_{\vec e}$ we denote the length of the part of the canonical horocycle between $\vec e$ and the interior edge of $P$ at the starting endpoint of $\vec e$.

\begin{lm}
\label{secder}
The functional $H$ is $C^2$ and we have
\begin{equation}
\label{dervw}
v\neq w: \frac{\partial^2 H}{\partial u_v\partial u_w}=\frac{\partial \kappa_v}{\partial u_w}=\frac{\partial \kappa_w}{\partial u_v}=\sum_{\vec e \in \vec E_{vw}(\mathcal T)}\frac{\cot(\alpha^+_{\vec e})+\cot(\alpha^-_{\vec e})}{\rho_{\vec e}\exp (u_v)\sin (l_{\vec e}) },
\end{equation}
\begin{equation}
\label{dervv}
v=w: \frac{\partial^2 H}{\partial^2 u_v}=\frac{\partial \kappa_v}{\partial u_v}=-\sum_{\vec e \in \vec E_{v}(\mathcal T)}\frac{\cot(\alpha^+_{\vec e})+\cot(\alpha^-_{\vec e})}{\rho_{\vec e}\exp (u_v)\sin (l_{\vec e}) }\cos(l_{\vec e})+\sum_{\vec e \in \vec E_{vv}(\mathcal T)}\frac{\cot(\alpha^+_{\vec e})+\cot(\alpha^-_{\vec e})}{\rho_{\vec e}\exp (u_v)\sin (l_{\vec e})}.
\end{equation}
\end{lm}

\begin{proof}
Similarly to the proof of Lemma~\ref{firstder} it is enough to compute the derivatives for the case when $u$ belongs to the interior of a cell $U(g, h, \mathcal T)$. In the general case one checks that the partial derivatives exist and are continuous, so the functional is $C^2$.

The computation follows from the computation of the derivatives of the dihedral angles in a single semi-ideal tetrahedron $\bar T=BA_1A_2A_3$. We inherit the notation from the previous proof and from Section~\ref{laws}. It remains to denote the lengths of the parts of the canonical horocycles in the ideal triangle $A_1A_2A_3$ by $\lambda_1$, $\lambda_2$, $\lambda_3$ respectively. 

The solid angle at the vertex $A_1$ of $\bar T$ cuts off a Euclidean triangle from the canonical horosphere at $A_1$. It has edge lengths $\lambda_1$, $\rho_{12}$ and $\rho_{13}$, and angles $\omega_1$, $\alpha_3$ and $\alpha_2$ respectively. We note that $\lambda_1$ does not change under our infinitesimal deformation, which takes place in a fixed discrete conformal class, so with fixed $(g, h)$. We have
$$\cos \omega_1=\frac{\rho_{12}^2+\rho_{13}^2-\lambda_1^2}{2\rho_{12}\rho_{13}}.$$
Differentiating it we get
$$\frac{\partial \omega_1}{\partial \rho_{12}}=-\frac{\cot(\alpha_3)}{\rho_{12}},~~~~~\frac{\partial \omega_1}{\partial \rho_{13}}=-\frac{\cot(\alpha_2)}{\rho_{13}}.$$
Now we differentiate the cosine law for the semi-ideal triangle $BA_1A_2$ given in Lemma~\ref{coslaw}, use the sine law from Lemma~\ref{sinlaw} (keeping in mind that the parameter $a_{12}$ is fixed) and get the derivatives
$$\frac{\partial \rho_{12}}{\partial u_1}=-\frac{\cot(l_{12})}{\exp(u_1)},~~~~~\frac{\partial \rho_{12}}{\partial u_2}=\frac{1}{\exp(u_1)\sin(l_{12})}.$$
This gives us
$$\frac{\partial\omega_1}{\partial u_1}=\frac{\partial \omega_1}{\partial \rho_{12}}\frac{\partial \rho_{12}}{\partial u_1}+\frac{\partial\omega_1}{\partial \rho_{13}}\frac{\partial\rho_{13}}{\partial u_1}=\frac{1}{\exp(u_1)}\left(\frac{\cot(\alpha_3)\cot(l_{12})}{\rho_{12}}+\frac{\cot(\alpha_2)\cot(l_{13})}{\rho_{13}}\right).$$
One can see that summing this up for all tetrahedra constituting $P$ gives us the first part of (\ref{dervv}). 

We also have
$$\frac{\partial\omega_1}{\partial u_2}=\frac{\partial \omega_1}{\partial \rho_{12}}\frac{\partial \rho_{12}}{\partial u_2}=-\frac{\cot(\alpha_3)}{\rho_{12}\exp(u_1)\sin(l_{12})}.$$
Summing this up for all tetrahedra gives us (\ref{dervw}) and the second part of (\ref{dervv}). The proof is finished.
\end{proof}

In particular, Lemma~\ref{secder} implies that $\kappa\circ\Psi_d$ is a $C^1$-map.

\section{Proof of Theorem~\ref{main}}

\subsection{Outline of the proof}

We fix $d\in\mathfrak M_{+, T}(S, V)$. Abusing the notation, in the rest of the paper we denote $\kappa\circ\Psi_{d}$ just by $\kappa$, so now
$$\kappa: U(d) \rightarrow (-\infty, 2\pi)^V$$
and it is a $C^1$-map.
Denote
$$
U_C(d):=\{u\in U(d):\text{for each $v \in V$, $\kappa_v(u)>0$}\}=\kappa^{-1}((0,2\pi)^V),
$$
which is an open domain in $\mathbb R^V$.
Denote $K$ as the set of $\kappa\in(0,2\pi)^V$ satisfying (\ref{e1}) and (\ref{e2}). Then $K$ is a convex open
domain in $\mathbb R^V$. Then Theorem~\ref{main} is equivalent to the fact that the discrete curvature map $\kappa$ gives a bijection from $U_C(d)$ to $K$.

It follows from Theorem~\ref{LT} that the image of $U_C(d)$ by $\kappa$ is in $K$. Then Theorem~\ref{main} is obtained from the following three lemmas.

\begin{lm}
\label{infrig}
The differential $d\kappa$ is non-degenerate on $U_C( d)$.
\end{lm}

\begin{lm}
\label{proper}
The map $\kappa:U_C( d)\rightarrow K$ is proper.
\end{lm}

\begin{lm}
\label{connect}
$U_C( d)$ is connected.
\end{lm}

\begin{proof}[Proof of Theorem~\ref{main}]
By Lemmas~\ref{infrig} and~\ref{proper}, $\kappa$ is a proper local homeomorphism on $U_C(d)$. Thus it is a covering map. Since $K$ is simply connected and by Lemma~\ref{connect} we get that $U_C(d)$ is connected, $\kappa$ is a global homeomorphism from $U_C(d)$ onto $K$.
\end{proof}

\subsection{Infinitesimal rigidity}

\subsubsection{Euclidean cone-polygons}

The differential $d\kappa$ is given by Lemma~\ref{secder}. It is however unclear to us how to use this expression in order to establish its non-degeneracy (in sharp contrast with the hyperbolic case where it is easy to see that such matrix is diagonally dominant, which implies its non-degeneracy, see, e.g., Corollary~5.4 in~\cite{Pro1}).
We will show the non-degeneracy of $d\kappa$ by a ``dimensional reduction'' argument appearing in~\cite{BI, Izm2}. We will deduce another formula for $d\kappa$, which allows to extract some additional information from the behavior of the discrete curvature of the horospherical links in the respective ideal cone-polyhedra. To this purpose we need to introduce new objects.

\begin{dfn}
A \emph{Euclidean cone-polygon} $Q$ is a complete metric space homeomorphic to the 2-disk with a locally Euclidean metric in the interior except one point $v$, where it is locally isometric to a Euclidean cone of total angle $\omega$, and with piecewise geodesic boundary.
\end{dfn}

We denote the kink points of the boundary of $Q$ by $w_1, \ldots, w_n$ in the cyclic order. The cone-polygon $Q$ can be naturally triangulated by geodesics from $v$ to the kink-points. By $\rho_i$ we denote the length of the geodesic from $v$ to $w_i$, by $\alpha_i$ we denote the angle at $w_i$, by $\lambda_i$ we denote the length of the geodesic segment between $w_i$ and $w_{i+1}$ and by $\omega_i$ we denote the angle between the geodesics from $v$ to $w_i$ and $w_{i+1}$.

The cone-polygon $Q$ is uniquely determined by the lengths $\rho_i$ and $\lambda_i$. 
Fix $\lambda_i$ and consider the point $\rho \in \R^n$. For small deformations of $\rho$ the cone-polygon remains to exist. Let $x \in T_\rho\R^n$ be a tangent vector. Every angle $\alpha_j$ is a smooth function around $\rho$. By $\dot \alpha_j(x)$ we denote its derivative in the direction $x$, which is a linear form in $x_i$. Similarly, $\dot\omega(x)$ is the directional derivative of $\omega$.

Consider the bilinear form $I$ on $T_\rho\R^n$ defined by
$$I(x, y)=\sum_i \rho_i x_i \dot\alpha_i(y)=\sum_{i,j}\rho_i\frac{\partial \alpha_i}{\partial\rho_j}x_iy_j.$$


\begin{lm}
\label{signature}
Let $\omega < 2\pi$. Then the signature of $I$ is $(+,-,\ldots,-)$.
\end{lm}

\begin{proof}
First we show the non-degeneracy. Suppose the converse, then there exists $x\in T_{\rho}\R^n$ such that for every other vector $y\in T_{\rho}\R^n$ we have $I(y, x)=0$. As the coefficients of $I$ are equal to $\rho_i\frac{\partial \alpha_i}{\partial \rho_j},$ we get $\dot \alpha_i(x)=0$ for each $i$.

Now cut $Q$ along the geodesic $vw_1$ and develop it on the Euclidean plane $\E^2$ as a polygon $Q'$. Slightly abusing the notation we denote the vertices of $Q'$ by $v, w'_1, w_2, \ldots, w_n, w''_1$. The infinitesimal deformation $x$ gives rise to an infinitesimal deformation of $Q'$. As $\dot \lambda_i(x)$ and $\dot\alpha_i(x)$ are zero, it is clear that the infinitesimal deformation of all vertices of $Q'$ except $v$ is induced by a global Killing field on $\E^2$. We can assume that it is zero. Then, as the infinitesimal change of the distance from $v$ to $w'_1$ is the same as of the distance from $v$ to $w''_1$, we see that the infinitesimal deformation of $v$ is orthogonal to the segment through $w'_1$ and $w''_1$. But every such non-zero deformation induces a non-trivial change on the sum of angles $\angle vw'_1w_2+\angle vw''_1w_n=\alpha_1$. Thus, it is zero and $x=0$. 

Now we deform $\rho_i$ and $\omega_i$ continuously so that $\rho_i$ remain positive and the sum of $\omega_i$ remains constant until all $\rho_i$ become equal and all $\omega_i$ become equal. As $I$ stays non-degenerate during this transformation, its signature does not change. It is easy to compute that in the end the matrix of $I$ becomes
$$
\frac{1}{\rho^2\sin\frac{\omega}{n} }
\begin{pmatrix}
-2\cos\frac{\omega}{n} & 1 & 0 & \ldots & 1\\
1 & -2\cos\frac{\omega}{n} & 1 & \ldots & 0\\
0 & 1 & -2\cos\frac{\omega}{n} & \ldots & 0\\
\vdots & \vdots & \vdots & \ddots & \vdots\\
1 & 0 & 0 & \ldots & -2\cos\frac{\omega}{n}
\end{pmatrix}.
$$ 
It has the spectrum $$\left\{\frac{2(\cos\frac{2\pi k}{n} -
\cos\frac{\omega}{n})}{\rho^2\sin\frac{\omega}{n}}, k = 0, \ldots, n-1\right\}.$$
This proves the claim on the signature.
\end{proof}

\begin{lm}
\label{link}
Assume that $\omega < 2\pi$ and all $\alpha_i<\pi$. If for $x \in T_\rho\R^n$ we have $\dot \omega(x)=0$, then $I(x, x)\leq 0$. Moreover, $I(x, x)= 0$ implies that $x=0$.
\end{lm}

\begin{proof}
Note that by the Gauss--Bonnet theorem we have
\begin{equation}
\label{t}
2\pi-\omega+\sum_i (\pi-\alpha_i)=2\pi.
\end{equation}

Consider the vector $x^0 \in T_\rho \R^n$ defined by $x^0_i:=1/\rho_i$. From (\ref{t}) we get
$$I(x^0,x^0)=\sum_i \dot\alpha_i(x^0)=-\dot\omega(x^0)=-\sum_i \dot \omega_i(x^0)=\sum_i \frac{\cot\alpha_i^++\cot\alpha_i^-}{\rho_i^2},$$
where $\alpha_i^+:=\angle vw_iw_{i+1}$ and $\alpha^-_i:=\angle vw_iw_{i-1}.$ As $\alpha_i^++\alpha_i^-=\alpha_i<\pi$, we see that $I(x^0,x^0)>0$. By Lemma~\ref{signature}, the positive subspace of the quadratic form associated to $I$ is 1-dimensional, so $x^0$ spans it. 

Now from $\dot\omega(x)=0$ and (\ref{t}) we have
$$0=\dot\omega(x)=-\sum_i\dot\alpha_i(x)=\sum_{i,j}\frac{\partial \alpha_i}{\partial \rho_j}x_j=I(x^0, x).$$
This means that $x$ either belongs to the negative subspace of the associated quadratic form to $I$ or is zero. This finishes the proof.
\end{proof}

\subsubsection{Infinitesimal rigidity of cone-polyhedra}

Now we return to the setting of our problem. Similarly to the previous subsection, if $x \in T_u \R^V$, then by $\dot \kappa_v(x)$, $\dot \rho_{\vec e}(x)$ or $\dot \alpha_{\vec e}(x)$ we denote the directional derivatives of various geometric quantities of $P(d,u)$, which are linear forms. By $\ddot H$ we denote the Hessian of $H$ considered as a bilinear form on $T_u\R^V$.

\begin{lm}
We have two expressions
\begin{equation}
\label{v1}
\ddot H(x, y)= \sum_{v \in V} x_v \dot \kappa_v(y),
\end{equation}
\begin{equation}
\label{v2}
\ddot H(x, y)= -\sum_{v \in V} e^{2u_v} \sum_{\vec e \in \vec E_v(\mathcal T)} \rho_{\vec e}\dot \rho_{\vec e}(x)\dot \alpha_{\vec e}(y).
\end{equation}
\end{lm}

\begin{proof}
The formula (\ref{v1}) is just Lemma~\ref{firstder} in a rewritten form. For (\ref{v2}) consider an oriented edge $\vec e$ emanating from $v$ and ending at $w$.

The horospherical link of $P(d, u)$ at a vertex $v$ is a Euclidean cone-polygon of total angle $\omega_v$ with angles $\alpha_{\vec e}$ for $\vec e \in \vec E_v(\mathcal T)$. Thus, we have
$$\dot \kappa_v(y)=-\dot\omega_v(y)=\sum_{\vec e \in \vec E_v(\mathcal T)} \dot\alpha_{\vec e}(y)$$
and we can rewrite (\ref{v1}) as
\begin{equation}
\label{v3}
\ddot H(x, y)=\sum_{e=\{vw\}\in E(\mathcal T)}\dot\alpha_{e}(y)(x_v + x_w).
\end{equation}

Consider the decorated semi-ideal triangle containing in $P(d, u)$ an edge $e$ with vertices $v$ and $w$. By $b_{vw}$, $b_{wv}$ we denote the lengths of the parts of $e$ in the triangle obtained by drawing a perpendicular from the marked point and by $\rho_{vw}$, $\rho_{wv}$ denote the respective parts of the canonical horocircles. Due to Lemma~\ref{perpend}, we have
$$\rho_{vw}^2=\exp(-2 b_{vw})-\exp(-2u_v),~~~~~\rho_{wv}^2=\exp(-2 b_{wv})-\exp(-2u_w).$$
Also $b_{vw}+b_{wv}=a_{e}$. Differentiating this we get 
\begin{equation}
\label{temp}
\exp(2b_{vw}-2u_v)\dot u_v+\exp(2b_{wv}-2u_w)\dot u_w-\exp(2b_{vw})\rho_{vw}\dot\rho_{vw}-\exp(2b_{wv})\rho_{wv}\dot\rho_{wv}=0.
\end{equation}

From Lemma~\ref{perpend} we also see that
$$\exp(2b_{vw}-2u_v)=\exp(2b_{wv}-2u_w).$$
So we can rewrite (\ref{temp}) as
$$
\dot u_v+\dot u_w-\exp(2u_v)\rho_{vw}\dot\rho_{vw}-\exp(2u_w)\rho_{wv}\dot\rho_{wv}=0.
$$ 
Considering the derivatives in the direction $x$ and substituting this to (\ref{v3}) we get exactly (\ref{v2}).
\end{proof}

Now we are ready to establish Lemma~\ref{infrig}.

\begin{proof}[Proof of Lemma~\ref{infrig}.]
Due to Lemma~\ref{firstder}, the Jacobian matrix $D\kappa=\left(\frac{\partial\kappa_v}{\partial u_w}\right)$ is the Hessian matrix $\ddot H$. We need to show its non-degeneracy.

Suppose the converse that for some $x$ and every $y$ we have $\ddot H(y, x)=0$. Due to (\ref{v1}), we obtain that $\dot\kappa_v(x)=0$ for all $v$. In particular, 
$$\ddot H(x, x)= \sum_{v \in V} x_v \dot \kappa_v(x)=0.$$
On the other hand, from Lemma~\ref{link} we get
$$\ddot H(x, x)= -\sum_{v \in V} e^{2u_v} \sum_{\vec e \in \vec E_v(\mathcal T)} \rho_{\vec e}\dot \rho_{\vec e}(x)\dot \alpha_{\vec e}(x) \geq 0$$
and its equality to zero means that all $\dot \rho_{\vec e}(x)$ are zero. One computes from the cosine laws of Lemma~\ref{coslaw} and the sine laws of Lemma~\ref{sinlaw} (see also the computations in the proof of Lemma~\ref{secder})
$$\dot \rho_{\vec e}(x)=\frac{x_v-\cos (l_{\vec e}) x_w}{\exp (u_v)\sin(l_{\vec e})}.$$
Clearly, from every point $v \in V$ there is at least one edge in a Delaunay triangulation $\mathcal T$. Also $\cos(l_{\vec e})<1$. Thus, all $\dot\rho_{\vec e}(x)=0$ implies that $x=0$.
\end{proof}

\subsection{Properness}

\subsubsection{Comparison lemma}

We will make use of the following lemma:

\begin{lm}
\label{weakangle}
There is no convex spherical cone-metric on the 2-disk $D$ with perimeter $2\pi$ and at least three kink-points at the boundary.
\end{lm}

Here by a convex spherical cone-metric on the 2-disk we mean that the conditions of Definition~\ref{conedfn} hold in the interior, the boundary is piecewise geodesic and the angle of each kink point is less than $\pi$.

For the proof of Lemma~\ref{weakangle} we need two classical facts.

\begin{thm}[Alexandrov's realization theorem]
\label{alex1}
Let $d$ be a convex spherical cone-metric on the 2-sphere $S$. Then there exists a convex polyhedron in the standard 3-sphere $\S^3$ with boundary isometric to $(S, d)$. 
\end{thm}

The proof of this theorem in the Euclidean case is given in~\cite[Section 4.3]{Ale}. The discussion of its adaptation to the spherical case is in~\cite[Section 5.3]{Ale}.

\begin{thm}[Alexandrov's rigidity theorem]
\label{alex2}
Let $\Omega_1$, $\Omega_2$ be two convex polyhedra in $\S^3$ and $f: \partial \Omega_1 \rightarrow \partial \Omega_2$ be an isometry. Then there exists an isometry $F: \S^3 \rightarrow \S^3$ inducing $f$.
\end{thm}

Similarly, this theorem is commonly proven in the Euclidean case, see~\cite[Section 3.3.2]{Ale}, but the proof works the same in the spherical case as noted in~\cite[Section 3.6.4]{Ale}.

\begin{proof}[Proof of Lemma~\ref{weakangle}]
Suppose that such metric $d$ exists. Double $(D, d)$ along the boundary and obtain the 2-sphere with a convex spherical cone-metric, which we denote by $\Omega=(S, d^S)$. By Theorem~\ref{alex1}, it admits a realization as a convex polyhedron $P$ in the standard 3-sphere $\S^3$.

Consider an orientation-inverting isometry $f: \Omega \rightarrow \Omega$ that changes the halves of $\Omega$. From Theorem~\ref{alex2} there exists an isometry $F: \S^3 \rightarrow \S^3$ inducing $f$ at $\partial P$. We see that $F$ is orientation-reversing and the set of its fixed points includes a closed curve $\tau \subset \partial P$, which is the boundary of two glued disks from $\Omega$, and is the set of the fixed points of $f$. We get that $F$ is a symmetry with respect to a geodesic 2-sphere, which we denote by $\S^2$.  Clearly, $\partial P \cap \S^2=\tau$. The broken line $\tau$ bounds a convex spherical polygon $\psi$, which has perimeter $2\pi$ and at least three kink points.

As $\psi$ is convex, it belongs to the spherical lune $\sigma$ determined by any its kink point. The perimeter of $\sigma$ is $2\pi$, which is the same as the perimeter of $\psi$. It is easy to see that the only option is that $\psi$ is a spherical lune itself. As $\psi$ has at least three kink points, this is a contradiction.
\end{proof}

\subsubsection{Proof of Lemma~\ref{proper}}

\begin{proof}[Proof of Lemma~\ref{proper}]
If not, there exists a sequence $\{u^k\}$, $k=1,2, \ldots$ in $U_C(d)$ such that $\{u^{k}\}$ is leaving any compact subset of $U_C(d)$ and $\kappa^{k}:=\kappa(u^{k})$ converge to some point $\bar\kappa\in K$. Denote $d^{k}:=\Psi_d(u^{k})$ and up to taking a subsequence assume that $u^{k}$ converge to some $\bar u\in[-\infty,+\infty]^V$.

As there are finitely many triangulations $\mathcal T$ of $(S, V)$ for which $\mathfrak D_+(\mathcal T)\cap \mathfrak C_+(d)$ is non-empty, after picking a subsequence we may assume that there exists a triangulation $\mathcal T$ of $(S,V)$ such that $d^{k}\in \mathfrak D_+(\mathcal T)$ for all $k$. Denote its edge lengths by $l^k \in (0, \pi)^{E(\mathcal T)}$.
It follows that for every edge $e \in E(\mathcal T)$ with endpoints $v$ and $w$ we have
$$\sin\left(\frac{l_e^k}{2}\right)=\exp\left(\frac{u^1_v+u^1_w-u^k_v-u^k_w}{2}\right)\sin\left(\frac{l_e^1}{2}\right),$$
see Remark~\ref{inverse}.

By picking a subsequence we may also assume that $l^{k}$ converge to some $\bar l\in [0,\pi]^{E(\mathcal T)}$, and all the inner angles of triangles of $\mathcal T$ converge. We have the following three possible cases:

(a) for some $v\in V$ we have $\bar u_v=-\infty$, or

(b) for all $w\in V$ we have $\bar u_w\neq-\infty$, but for some $v \in V$ we have $\bar u_v = +\infty$, or

(c) all $\bar u_v \in (-\infty, +\infty)$.




\textit{Case (a).} Clearly, for every $w$ incident to $v$ we have $\bar u_w = +\infty$. Thus, in every triangle adjacent to $v$ the length of the opposite edge is zero in the limit. We need to consider further two subcases.

\textit{Case (a.1): For every edge $e$ incident to $v$ we have $\bar l_e \neq \pi$.}

Let $T$ be a triangle adjacent to $v$ and $\omega^{k}_{v, T}$ be the angle of $T$ at $v$ in $d^{k}$. We claim that $\omega^{k}_{v, T} \rightarrow 0$. Indeed, if the limiting length $\bar l_e$ of an edge $e$ of $T$ incident to $v$ is not zero, then the claim is clear. If not, denote the other endpoint of $e$ by $w$, the other edge of $T$ incident to $w$ by $e'$, the other end of $e'$ by $w'$ and the angle of $T$ at $w$ in $d^k$ by $\omega^k_{w, T}$. Then
$$\lim_k \sin\omega^{k}_{v, T}=\lim_k\frac{\sin l^k_{e'}}{\sin l^k_e}\sin\omega^k_{w, T}=\lim_k\frac{\sin\left(l^k_{e'}/2\right)}{\sin\left(l^k_e/2\right)}\sin\omega^k_{w, T}=$$
$$=\lim_k\exp\left(\frac{u^1_{w'}-u^k_{w'}-u^1_v+u^k_v}{2}\right)\frac{\sin\left(l^1_{e'}/2\right)}{\sin\left(l^1_e/2\right)}\sin\omega^k_{w, T}=0.$$
Also one can see that for sufficiently large $k$ the edge $e'$ is the smallest in $T$. Thus, $\omega^{k}_{v, T} \rightarrow 0$. We get $\bar\kappa_v=2\pi$, which contradicts to $\bar\kappa \in K$.

\textit{Case (a.2): For some edge $e$ incident to $v$ we have $\bar l_e = \pi$.} In this case, if $e'$ is another edge adjacent to $v$ in the same triangle with $e$, then by the triangle inequality we get $\bar l_{e'}=\pi$. By induction we see that this holds for every edge incident to $v$. Let $T$ be a triangle adjacent to $v$ and $\bar \omega^{k}_{v, T}$ be the angle of $T$ at $v$ in the limit. Then
$$\area(T, d^{k})\rightarrow 2\bar\omega^{k}_{v,T}$$
and
$$
\sum_{T: v \in T} \area(T, d^{k})\rightarrow 2\cdot(2\pi-\bar\kappa_v).
$$
Then by the Gauss--Bonnet theorem
$$
\left(\sum_{w\neq v}\bar\kappa_w\right)-\bar\kappa_v=4\pi-2\bar\kappa_v-\lim_k \area(S, d^{k})\leq0.
$$
This contradicts to $\bar\kappa\in K$.

\textit{Case (b).} If all $\bar u_w=+\infty$, then $\area(S,d^{k})\rightarrow 0$ and by the Gauss--Bonnet theorem we get $\sum_{w\in V}\bar\kappa_w=4\pi$. This contradicts to $\bar\kappa\in K$.

Suppose that there exist adjacent vertices $w$ and $v$ such that $\bar u_v=+\infty$ and $\bar u_w\in(-\infty,\infty)$.
Then by the triangle inequality and induction we get $\bar u_{v'}=+\infty$ for any $v'$ adjacent to $w$. It follows that for every $w'$ such that $\bar u_{w'} \in (-\infty, \infty)$ and for every $v'$ adjacent to $w'$ we have $\bar u_{v'} =  +\infty$. Thus, $\area(S,d^{k})\rightarrow 0$, which, again, contradicts to $\bar\kappa\in K$.

\textit{Case (c).} As all $\bar u_v \neq +\infty$, for each $e \in E(\mathcal T)$ we have $\bar l_e \neq 0$. Note also that if for a triangle $T$ the limit shape degenerates to an arc, then the Delaunay condition implies that the triangle $T'$ adjacent to $T$ by the largest side of $T$ degenerates to the half-sphere in the limit. In any case we see that the metrics $d^k$ converge to a convex spherical cone-metric $\bar d$ on $(S, V)$ with strictly positive curvature at every point of $V$. 

We need to examine what happens if a triangle $T$ of $\mathcal T$ converges in $\bar d$ to a non-convex one, i.e., to a spherical lune. If this does not happen, then the argument above shows that also all triangle inequalities stay strict, thus $\bar d \in \mathfrak D_+(\mathcal T)$ and Lemma~\ref{proper} is proved. Consider two subcases.

\textit{Case (c.1): no edge of $T$ equals $\pi$ in $d$.} Then $T$ converges to a half-sphere in $\bar d$. Its vertices can not coincide, as otherwise we get a vertex of non-positive curvature. Then $\overline{S- T}$ with the limit metric is a disk with convex spherical cone-metric, perimeter $2\pi$ and three kink points at the boundary. This contradicts to Lemma~\ref{weakangle}.

\textit{Case (c.2): some edge $e$ of $T$ equals $\pi$.} Let $T'$ be the triangle adjacent to $T$ by $e$. As the other edges of $T$ are strictly smaller than $\pi$, such $T'$ exists. Then $T'$ is also a spherical lune in $\bar d$. Moreover, due to the Delaunay condition, the union $T \cup T'$ is a spherical lune that is at least a half-sphere. The vertices of $T \cup T'$ can not coincide, as otherwise we get a vertex of non-positive curvature. Then $\overline{S-T-T'}$ with the limit metric is a disk with convex spherical cone-metric, perimeter $2\pi$ and four kink points at the boundary. This contradicts to Lemma~\ref{weakangle}.

\end{proof}

\subsection{Connectivity}

For $u \in U(d)$ we denote the metric $\Psi_d(u)$ by $u*d$.
Let $\mathbf 1=(1,...,1)\in\mathbb R^V$. It is easy to see that if $u \in U(d)$, then for all $\lambda > 0$ we have $u+\lambda \mathbf 1 \in U(d)$. Moreover, if $u*d \in \mathfrak D_+(\mathcal T)$, then $(u+\lambda \mathbf 1)*d \in \mathfrak D_+(\mathcal T)$.

The Euclidean cone-metric $\Phi((u+\lambda \mathbf 1)*d)$ is obtained from $\Phi(u*d)$ by multiplication of all edges in any geodesic triangulation by $e^{-\lambda}$. Thus, $$\delta(u+\lambda \mathbf 1)=\delta(u)$$
where by $\delta(u) \in \R^V$ we denote the discrete curvature of the Euclidean cone-metric $\Phi(u*d)$.

It is also not hard to see

\begin{lm}
\label{contract2}
For every $u\in U(d)$ and $v \in V$
$$\lim_{\lambda \rightarrow +\infty} \kappa_v(u+\lambda \mathbf 1) = \delta_v (u).$$
\end{lm}
\begin{proof}
Let $T^+_0$ be a convex spherical triangle, $v$ be its vertex and $\omega^+_0$ be its angle at $v$. Consider $T^+_0$ embedded in the unit 2-sphere in the Euclidean 3-space and let $T^0_0$ be the subtending Euclidean triangle. By $\omega$ denote its angle at $v$. By $T^0_{\lambda}$ denote the Euclidean triangle obtained from $T^0_0$ by multiplication of all its side length by $e^{-\lambda}$, $\lambda \geq 0$. We consider the vertices of $T^0_{\lambda}$ still at the unit sphere, $v$ still being its vertex. By $T^+_{\lambda}$ we denote the respective spherical triangle subtended by $T^0_{\lambda}$, by $\omega^+_{\lambda}$ denote its angle at $v$. It is enough to show that
$$\lim_{\lambda \rightarrow +\infty}\omega^+_{\lambda}=\omega.$$

Let $\beta_{1, \lambda}$, $\beta_{2, \lambda}$ be the angles between sides of $T^0_{\lambda}$ at $v$ and the ray from $v$ to the origin. In the spherical link at $v$ we have the spherical triangle with sides $\beta_{1, \lambda}$, $\beta_{2, \lambda}$ and $\omega$. The angle of this triangle opposite to the side of length $\omega$ is equal to $\omega^+_{\lambda}$. Clearly, $\beta_{1, \lambda}, \beta_{2, \lambda}$ tend to $\pi/2$ as $\lambda \rightarrow +\infty$. This shows the desired claim.
\end{proof}

We need another simple statement

\begin{lm}
\label{contract1}
For every $u\in U(d)$ and $v \in V$
$$\dot \kappa_v(\mathbf 1)= \sum_{w \in V}\frac{\partial \kappa_v}{\partial u_w} > 0.$$
Here we consider $\mathbf 1$ as an element of $T_u\R^V$ and $\dot \kappa_v(\mathbf 1)$ is the directional derivative in this direction.
\end{lm}

\begin{proof}
Indeed, due to Lemma~\ref{secder} we have
$$\sum_{w \in V}\frac{\partial \kappa_v}{\partial u_w}=\sum_{\vec e \in E_{v}(\mathcal T)}\frac{\cot(\alpha^+_{\vec e})+\cot(\alpha^-_{\vec e})}{\rho_{\vec e}\exp (u_v)\sin (l_{\vec e}) }\left(1-\cos(l_{\vec e})\right)>0.$$
\end{proof}

\begin{proof}[Proof of Lemma~\ref{connect}]
Let $u$, $u' \in U_C(d)$. Due to Theorem~\ref{glsw}, there exists a path of metrics $d_t \subset \mathfrak C_0(\Psi(d))$, $t \in [0,1]$, connecting $d_0 := \Psi(u*d)$ with $d_1 := \Psi(u'*d)$ such that
$$\delta(d_t)=t\delta(u')+(1-t)\delta(u).$$

Moreover, as $\delta$ is constant on the homothetic changes of the metric, one can choose $d_t$ so that the radii of the circumscribed circles of all triangles in the Delaunay triangulations of $d_t$ are less than one. Thus, there exists the preimage $\Psi^{-1}(d_t) \subset \mathfrak C_+(d)$ descending to a path $u_t \subset U(d)$ connecting $u$ and $u'$. 

From Lemma~\ref{contract2} for a sufficiently large $\lambda$ the path $u_{\lambda, t}:=u_t+\lambda \mathbf 1$
belongs to $U_C(d)$. Lemma~\ref{contract1} implies that we can connect in $U_C(d)$ the endpoints of this path $u_{\lambda, t}$ with $u$ and $u'$ respectively along the direction $\mathbf 1$. This finishes the proof.
\end{proof}

\bibliographystyle{abbrv}
\bibliography{ref}

\end{document}